\documentclass{amsart}
\usepackage{hyperref}
\newcommand*{\mailto}[1]{\href{mailto:#1}{\nolinkurl{#1}}}
\newcommand{\arxiv}[1]{\href{http://arxiv.org/abs/#1}{arXiv:#1}}

\newtheorem{theorem}{Theorem}[section]
\newtheorem{definition}[theorem]{Definition}
\newtheorem{lemma}[theorem]{Lemma}
\newtheorem{corollary}[theorem]{Corollary}
\newtheorem{remark}[theorem]{Remark}
\newtheorem{example}[theorem]{Example}
\newtheorem{hypothesis}[theorem]{Hypothesis {\bf H.}\hspace*{-0.6ex}}

\newcommand{\R}{{\mathbb R}}
\newcommand{\N}{{\mathbb N}}
\newcommand{\Z}{{\mathbb Z}}
\newcommand{\C}{{\mathbb C}}

\newcommand{\nn}{\nonumber}
\newcommand{\be}{\begin{equation}}
\newcommand{\ee}{\end{equation}}

\newcommand{\ti}{\tilde}

\newcommand{\spr}[2]{\left\langle #1,#2 \right\rangle}
\newcommand{\dpr}[3][\hr]{\left( #2 , #3 \right)_{#1}}

\newcommand{\E}{\mathrm{e}}
\newcommand{\I}{\mathrm{i}}
\newcommand{\gk}{\kappa}

\newcommand{\im}{\mathrm{Im}}
\newcommand{\re}{\mathrm{Re}}
\newcommand{\dom}{\mathfrak{D}}

\newcommand{\fdom}{\mathfrak{Q}}
\newcommand{\hr}{\mathfrak{H}}

\DeclareMathOperator{\lspan}{span}

\newcommand{\floor}[1]{\lfloor#1 \rfloor}
\newcommand{\ceil}[1]{\lceil#1 \rceil}

\newcommand{\eps}{\varepsilon}

\newcommand{\sig}{\sigma}
\newcommand{\lam}{\lambda}
\newcommand{\gam}{\gamma}

\numberwithin{equation}{section}

\begin{document}

\title[Singular $m$-Functions of Schr\"odinger Operators]{On the Singular Weyl--Titchmarsh Function of
Perturbed Spherical Schr\"odinger Operators}

\author[A.\ Kostenko]{Aleksey Kostenko}
\address{Institute of Applied Mathematics and Mechanics\\
NAS of Ukraine\\ R. Luxemburg str. 74\\
Donetsk 83114\\ Ukraine\\ and School of Mathematical Sciences\\
Dublin Institute of Technology\\
Kevin Street\\ Dublin 8\\ Ireland}
\email{\mailto{duzer80@gmail.com}}

\author[G.\ Teschl]{Gerald Teschl}
\address{Faculty of Mathematics\\ University of Vienna\\
Nordbergstrasse 15\\ 1090 Wien\\ Austria\\ and International
Erwin Schr\"odinger
Institute for Mathematical Physics\\ Boltzmanngasse 9\\ 1090 Wien\\ Austria}
\email{\mailto{Gerald.Teschl@univie.ac.at}}
\urladdr{\url{http://www.mat.univie.ac.at/~gerald/}}

\thanks{J. Diff. Eq. {\bf 250}, 3701--3739 (2011)}
\thanks{{\it Research supported by the Austrian Science Fund (FWF) under Grant No.\ Y330}}


\begin{abstract}
We investigate the singular Weyl--Titchmarsh $m$-function of perturbed spherical Schr\"odinger operators (also
known as Bessel operators) under the assumption that the perturbation $q(x)$ satisfies $x q(x) \in L^1(0,1)$.
We show existence plus detailed properties of a fundamental system of solutions which are entire with respect
to the energy parameter. Based on this we show that the singular $m$-function belongs to the generalized
Nevanlinna class and connect our results with the theory of super singular perturbations.
\end{abstract}

\maketitle

\section{Introduction}
\label{sec:int}

In this paper we will investigate perturbed spherical Schr\"odinger operators (also known as Bessel operators)
\be\label{eq-bes}
\tau=-\frac{d^2}{dx^2}+\frac{l(l+1)}{x^2}+q(x),\quad l\ge -\frac{1}{2},\qquad x\in\R_+:=(0,+\infty),
\ee
where the potential $q$ is real-valued satisfying
\begin{equation}\label{q:hyp}
q \in L^1_{\mathrm{loc}}(\R_+), \qquad \begin{cases} x\, q(x) \in L^1(0,1), & l>-\frac{1}{2},\\
x(1-\log(x)) q(x) \in L^1(0,1), & l = -\frac{1}{2}. \end{cases}
\end{equation}
Note that we explicitly allow non-integer values of $l$ such that we also cover the case of
arbitrary space dimension $n\ge 2$, where $l(l+1)$ has to be replaced by $l(l+n-2) + (n-1)(n-3)/4$ \cite[Sec.~17.F]{wdln}.
Due to its physical importance this equation has obtained much attention in the past
and we refer for example to \cite{ahm}, \cite{gr}, \cite{kst}, \cite{se}, \cite{wdln} and the references therein.

We will use $\tau$ to describe the formal differential expression and
$H$ the self-adjoint operator acting in $L^2(\R_+)$ and given by $\tau$ together with the usual boundary condition at $x=0$:
\be
\lim_{x\to0} x^l ( (l+1)f(x) - x f'(x))=0, \qquad l\in[-\frac{1}{2},\frac{1}{2}).
\ee
We are mainly interested in the case where $\tau$ is limit point at $\infty$, but if it is not, we simply
choose another boundary condition there. Moreover, one could also replace $\R_+$ by a bounded
interval $(0,b)$.

If $l=0$ and $q \in L^1(0,1)$ such that the left endpoint is regular, it is well known
that one can associate a single function $m(z)$, the Weyl--Titchmarsh (or Weyl) $m$-function, with $H$,
such that $m(z)$ contains all the information about $H$. In the general case (in particular when $l\ge \frac{1}{2}$ and
 $\tau$ is limit point at the left endpoint) it was shown only recently that one can still introduce a
singular Weyl function $M(z)$ which serves a similar purpose (we refer to
Gesztesy and Zinchenko \cite{gz}, Fulton and Langer \cite{ful08}, \cite{fl_09}, Kurasov and Luger \cite{KurLug}, Derkach \cite{Der98}, and Dijksma and Shondin \cite{DSh_00}).
For a comprehensive treatment we refer to our recent work with Sakhnovich \cite{kst2}.

The key ingredient for defining a Weyl $m$-function is an entire system of linearly independent
solutions $\phi(z,x)$, $\theta(z,x)$ of the underlying differential equation $\tau u = z u$, $z\in\C$,
normalized such that the Wronskian $W(\theta(z),\phi(z))$ equals one. To make the connection with $H$,
one solution, say $\phi(z,x)$, has to be chosen such that it lies in the domain of $H$ near the endpoint $x=0$
(i.e., $\phi(z,.)\in L^2(0,1)$ and it satisfies the boundary condition at $x=0$ if $H$ is limit point at $x=0$).
Once $\phi(z,x)$ and $\theta(z,x)$ are given, the Weyl $m$-function $M(z)$ can be defined by the requirement that the solution
\be\label{eq:i_02}
\psi(z,x) = \theta(z,x) + M(z) \phi(z,x)
\ee
is in the domain of $H$ near $+\infty$, i.e., $\psi(z,.)\in L^2(1,+\infty)$.

While this prescription sounds rather straightforward, it has turned out to be rather subtle! Namely, the following problems
naturally arise in the study of singular $m$-functions:
\begin{itemize}
\item existence of entire solutions $\phi(z,x)$ and $\theta(z,x)$ as above.
\item analytic properties of the singular $m$-function.
\item a canonical normalization of the fundamental solutions $\phi$ and $\theta$ at a singular endpoint $x=0$.
\end{itemize}
In \cite{kst2} we have shown that a necessary and sufficient condition for a system of solutions $\phi(z,x)$ and $\theta(z,x)$ to exist is that
one operator (and hence all) associated with $\tau$ restricted to a vicinity of the singular endpoint has purely discrete spectrum.
This clearly affirmatively settles the first question. In addition, it implies that the corresponding singular $m$-function \eqref{eq:i_02}
is analytic in the entire upper (and hence lower) half plane and thus also partly settles the second question. Moreover, we have shown
that there exists a renormalization of the fundamental solutions such that the corresponding singular Weyl function is a generalized
Nevanlinna or even Herglotz--Nevanlinna function. However, the corresponding choice of fundamental solutions is not naturally given
and it was only indirectly constructed.

On the other hand, in the special case of Bessel operators \eqref{eq-bes}, under the additional assumption that the potential $q(x)$
is analytic and of Fuchs type near $x=0$, there is a natural choice of fundamental solutions, namely those obtained from the Frobenius
method. It was shown by Fulton and Langer \cite{fl_09} that this choice leads to a singular Weyl function in the generalized
Nevanlinna class $N_\kappa^\infty$, where $\kappa \le \gk_l:=\floor{\frac{l}{2}+\frac{3}{4}}$ (for the definition of $N_\kappa^\infty$
see Appendix~\ref{app:nkappa}). Here $\floor{x}= \max \{ n \in \Z | n \leq x\}$ is the usual floor function.
Moreover, for the Coulomb case $q(x) = q_0/x$, Kurasov and Luger \cite{KurLug} proved that in
fact $\kappa = \gk_l$ (see also \cite{DSh_00}, where the case $q\equiv 0$ was treated). Our approach from \cite{kst2} applied to
\eqref{eq-bes} with potential $q(x)$ satisfying \eqref{q:hyp} shows that
there is a choice of fundamental solutions such that the singular Weyl function is in $N_\kappa^\infty$ with $\kappa \le \ceil{\frac{l+1}{2}}$. Here $\ceil{x}= \min \{ n \in \Z | n \geq x\}$ is the ceil function.
However, if $q(x)$ it not analytic (at least near $x=0$) there is no natural choice since the Frobenius method breaks down in this case.
It is the aim of the present paper to give a characterization of the fundamental solutions which lead to a singular Weyl function in
$N_{\gk_l}^\infty$ thereby extending the results from \cite{fl_09} and \cite{KurLug} to the class \eqref{q:hyp}.

Our approach is based on two main ingredients:
\begin{enumerate}
\item a detailed analysis of solution of the underlying differential equation and
\item the theory of super singular perturbations \cite{be}, \cite{Der98}, \cite{DHdS}, \cite{DKSh_05}, \cite{DLSZ}, \cite{DSh_00}, \cite{Ku}, \cite{sho} (see also Appendix \ref{app:ssp}).
\end{enumerate}
More precisely, in Section~\ref{sec:ass_sol} we show that for $l>-\frac{1}{2}$ real entire solutions $\phi(z,x)$ and $\theta(z,x)$ can be chosen to satisfy
the following "asymptotic normalization" at $x=0$ (Lemma~\ref{lem_ass})
\be\label{eq:i.03}
\phi(z,x)=x^{l+1}(1+o(1)),\quad
\theta(z,x)= \begin{cases} x^{-l}\left(\frac{1}{2l+1}+o(1)\right), & l>-\frac{1}{2},\\ - x^{1/2}\log(x)(1+o(1)), & l= -\frac{1}{2}, \end{cases}
\quad x\to 0.
\ee
Note that, while the first solution $\phi(z,x)$ is unique under this normalization, the second solution $\theta(z,x)$ is not, since for any entire $f(z)$ the new solution
$\widetilde{\theta}(z,x)=\theta(z,x)+f(z)\phi(z,x)$ also satisfies \eqref{eq:i.03}. So, we need an additional normalization assumption for $\theta(z,x)$.
To this end we show that there is a Frobenius type representation for $\phi(z,x)$ and $\theta(z,x)$ (see Lemmas~\ref{lemphi} and \ref{lem:theta_fr})
and the corresponding normalization is given in Definition~\ref{def_thetafrob} (see also Corollary~\ref{cor:3.7}): we will call $\theta$ a
\emph{Frobenius type solution} if
\be\label{eq:i.04}
\lim_{x\to 0}W_x\big(\theta^{(n_l+1)}(z),\theta(z_0)\big)\equiv 0,\qquad n_l:=\floor{l+1/2},
\ee
where $W_x(f,g)=f(x)g'(x)-f'(x)g(x)$ is the usual Wronskian.
Note that such a $\theta(z,x)$ always exists since by item {\it (vi)} of Corollary~\ref{cor:wronski} this limit exists and is a real
entire function in $z$ (let us denote it by $F^{(n_l+1)}(z)$). Therefore, $\widetilde{\theta}(z,x)=\theta(z,x)-F(z)\phi(z,x)$ satisfies the above assumption.

Furthermore, the Frobenius type representation of the fundamental solutions enables us to apply the theory of super singular perturbations.
The connection between the Weyl--Titchmarsh theory for Sturm--Liouville operators and the theory of singular perturbations is well known and goes
back to the pioneering work of Mark Krein on extension theory (see, e.g., \cite{si_95}). Thus, in the regular case $l=0$ and $q\in L^1(0,1)$, the Weyl--Titchmarsh
function, which corresponds to Neumann boundary condition at $x=0$, can be considered as a $Q$-function of the operator $H$,
\[
m_N(z)=\dpr[L^2]{\delta}{(H-z)^{-1}\delta}, \quad z\in \rho(H),
\]
where $\delta$ is the Dirac delta distribution and the inner product is understood as a pairing between $W^{1,2}(\R_+)$ and $W^{-1,2}(\R_+)$
(for the details see Example~\ref{ex:c.01} and also \cite[\S I.6]{si_95}).

To introduce the $Q$-function for $H$ in the case $l\ge\frac{1}{2}$, i.e., in the limit point case at $x=0$, one needs the theory of super singular perturbations \cite{be}, \cite{DHdS}, \cite{DKSh_05}, \cite{DLSZ}, \cite{Ku}, \cite{sho}. Moreover, it was first observed in \cite{DSh_00}, \cite{KurLug} that this $Q$-function is closely connected with the singular $m$-function \eqref{eq:i_02} (note that  in \cite{Der98} Derkach introduced the singular $m$-function for Laguerre operators). For instance (see also Section \ref{sec:exmpl} below), for $q\equiv 0$, it is shown \cite{DSh_00, KurLug} that the (maximal) self-adjoint operator $H_l$ associated with $\tau_l$, $\tau_l:=\tau$ if $q\equiv 0$, can be realized as an $\mathfrak{H}_{-n_l-2}$-perturbation and one of the corresponding Weyl functions $M_l(z)$ is given by \eqref{eq:II.11} below. Also, in this case $M_l\in N^\infty_{\gk_l}$, where $\gk_l=\floor{\frac{l}{2}+\frac{3}{4}}$.
Moreover, the perturbation element $\varphi$ is $\varphi_l:=(\widetilde{H}_l-z)\psi_l(x,z)$, where $\psi_l(x,z)=\frac{\I\pi}{2}\E^{\I\frac{2l+1}{4}\pi}(-z)^{\frac{2l+1}{4}}\sqrt{x}H_{l+\frac{1}{2}}^{(1)}(\I x\sqrt{-z})$, and $\widetilde{H}_l$ is an $[\mathfrak{H}_{-n_l},\mathfrak{H}_{-(n_l+2)}]$ continuation of $H_l$. Here $H_\nu^{(1)}$ denotes the
Hankel function of order $\nu$ of the first kind.

Lemmas~\ref{lem_5.5} and \ref{lem:theta_fr} enable us to extend the above scheme to general Bessel operators with potentials satisfying \eqref{q:hyp}.
Namely, Lemmas~\ref{lem_5.5} and \ref{lem:theta_fr} allow us to conclude that the solution $\psi(z,x)$ defined by \eqref{eq:i_02} and \eqref{eq:i.03} satisfies
\be\label{eq:i.05}
\partial^{\gk_l}_z\psi(z,x)=\gk_l!(\widetilde{H}-z)^{-\gk_l}\psi(z,x)\in L^2(\R_+),\qquad \partial^{\gk_l-1}_z \psi(z,x)\notin L^2(0,1).
\ee
Therefore, setting
\[
\varphi:=(\widetilde{H}-\I)\psi(\I,x)\in \hr_{-2(\gk_l+1)}\setminus \hr_{-2\gk_l},
\]
we can introduce the $Q$-function $\widetilde{M}(z)$ for the operator $H$ via \eqref{eq:c.14}--\eqref{eq:c.15}.
In Section~\ref{sec:m-func} we will then show that the singular Weyl function \eqref{eq:i_02} and the $Q$-function $\widetilde{M}$ are connected by the following relation
\[
M(z)=\widetilde{M}(z)+G(z),
\]
where the function $G(z)$ is entire (Theorem \ref{th:main}).
Moreover, we show that $G(z)$ is a real polynomial of order at most $2\gk_l+1$ if $\theta(z,x)$ is a Frobenius type solution,
that is, $\theta(z,x)$ satisfies condition \eqref{eq:i.04}.

To conclude, we briefly describe the content of the paper. In Section~\ref{sec:exmpl} we consider the unperturbed Bessel operator.
The next section deals with the properties of a fundamental system of solutions, which are entire with respect to the energy parameter.
In particular, we prove a Frobenius type representation for the fundamental solutions. In Section~\ref{sec:m-func} we prove our main
result, Theorem \ref{th:main}.

Appendix \ref{app:hi} contains necessary information on Hardy type inequalities, which we need in Section~\ref{sec:ass_sol}.
We also collect necessary information on generalized Nevanlinna functions and the theory of super singular perturbations in
Appendices~\ref{app:nkappa} and \ref{app:ssp}, respectively.

\section{An example}
\label{sec:exmpl}

We begin our investigations by discussing the prototypical example:
The spherical Schr\"odinger equation given by
\be
H_l = - \frac{d^2}{dx^2} + \frac{l(l+1)}{x^2}, \qquad x\in(0,+\infty), \: \ l \ge -\frac{1}{2},
\ee
with the usual boundary condition at $x=0$ (for $l\in[-\frac{1}{2},\frac{1}{2})$)
\be
\lim_{x\to0} x^l ( (l+1)f(x) - x f'(x))=0.
\ee
Two linearly independent solutions of the underlying differential equation
\be
- u''(x) + \frac{l(l+1)}{x^2} u(x) = z u(x)
\ee
are given by
\be\label{defphil}
\phi_l(z,x) = C_l^{-1}z^{-\frac{2l+1}{4}} \sqrt{\frac{\pi x}{2}} J_{l+\frac{1}{2}}(\sqrt{z} x),\quad C_l:=\frac{\Gamma(l+\frac{3}{2}) 2^{l+1}}{\sqrt{\pi}},
\ee
\be\label{defthetal}
\theta_l(z,x) = -C_lz^{\frac{2l+1}{4}} \sqrt{\frac{\pi x}{2}} \begin{cases}
\frac{-1}{\sin((l+\frac{1}{2})\pi)} J_{-l-\frac{1}{2}}(\sqrt{z} x), & {l+\frac{1}{2}} \in \R_+\setminus \N_0,\\
Y_{l+\frac{1}{2}}(\sqrt{z} x) -\frac{1}{\pi}\log(z) J_{l+\frac{1}{2}}(\sqrt{z} x), & {l+\frac{1}{2}} \in\N_0,\end{cases}
\ee
where $J_{l+\frac{1}{2}}$ and $Y_{l+\frac{1}{2}}$ are the usual Bessel and Neumann functions \cite{as}.
All branch cuts are chosen along the negative real axis unless explicitly stated otherwise.
If $l$ is an integer they of course reduce to spherical Bessel and Neumann functions
and can be expressed in terms of trigonometric functions (cf.\ e.g.\ \cite{as}, \cite[Sect.~10.4]{tschroe})

Using the power series for the Bessel and Neumann functions one verifies that they have the form
\be\label{eq:2.06}
\phi_l(z,x) =   x^{l+1} \sum_{k=0}^\infty\frac{C_{l,k}^\phi x^{2k}}{k!}z^k,\qquad
C_{l,k}^\phi=\frac{(-1)^k}{4^k(l+\frac{3}{2})_k},
\ee
\be\label{eq:2.07}
\theta_l(z,x) =  \begin{cases} \frac{x^{-l}}{2l+1}
 \sum\limits_{k=0}^\infty\frac{C_{l,k}^\theta x^{2k}}{k!}z^k, & l+\frac{1}{2} \in \R_+\setminus \N_0,\\
 \frac{x^{-l}}{2l+1}
\left(\sum\limits_{k=0}^{n_l-1}\frac{C_{l,k}^\theta x^{2k}}{k!}z^k -\frac{\log(x/2)}{4^l(n_l-1)!}\sum\limits_{k=n_l}^\infty\frac{C_{l,k-n_l}^\phi x^{2k}}{k!}z^{k} \right. \\
\left.\qquad \quad+
\sum\limits_{k=n_l}^\infty\frac{C_{l,k} x^{2k}}{k!}z^{k}\right),
& l+\frac{1}{2} \in\N_0,
\end{cases}
\ee
where $n_l=\floor{l+\frac{1}{2}}$ and
\be\label{eq:2.09}
C_{l,k}^\theta=\begin{cases}
\frac{(-1)^k}{4^k(-l+\frac{1}{2})_k}, & l+\frac{1}{2} \notin\N_0,\\
\frac{1}{4^k(l-k+\frac{1}{2})_k}, & l+\frac{1}{2} \in\N_0,
\end{cases}\qquad C_{l,k}=(-1)^k\frac{\psi(k+1)+\psi(k-n_l+1)}{4^{k}k!},
\ee
and $\psi(\cdot)$ is the psi-function \cite[(6.3.2)]{as}.
Here we have used the Pochhammer symbol
\be\label{eq:2.10}
(x)_0=1, \quad (x)_j= x (x+1) \cdots (x+j-1) = \frac{\Gamma(x+j)}{\Gamma(x)}.
\ee
In particular,  both functions are entire with respect to $z$ and according to \cite[(9.1.16)]{as} their Wronskian is given by
\be \label{a0}
W(\theta_l(z),\phi_l(z))=1.
\ee
Moreover, on $(0,\infty)$ and $l\ge -1/2$ we have
\be
\psi_l(z,x)= \theta_l(z,x) + M_l(z) \phi_l(z,x)= \I C_l\sqrt{\frac{\pi x}{2}} (\I \sqrt{-z})^{l+\frac{1}{2}} H_{l+\frac{1}{2}}^{(1)}(\I\sqrt{-z}x)
\ee
with
\be\label{eq:II.11}
M_l(z) = \begin{cases}
\frac{-C_l^2}{\sin((l+\frac{1}{2})\pi)} (-z)^{l+\frac{1}{2}}, & {l+\frac{1}{2}}\in\R_+\setminus \N_0,\\
\frac{-C_l^2}{\pi} z^{l+\frac{1}{2}}\log(-z), & {l+\frac{1}{2}} \in\N_0,\end{cases}\qquad C_l=\frac{\Gamma(l+\frac{3}{2}) 2^{l+1}}{\sqrt{\pi}},
\ee
where all branch cuts are chosen along the negative real axis and $H_{l+1/2}^{(1)}(z) = J_{l+1/2}(z) + \I Y_{l+1/2}(z)$ is the Hankel
functions of the first kind. The associated spectral measure is given by
\be
d\rho_l(\lam) = C_l^2\chi_{[0,\infty)}(\lam) \lam^{l+\frac{1}{2}} \frac{d\lam}{\pi},  \qquad l \geq -\frac{1}{2},
\ee
and the associated spectral transformation is just the usual Hankel transform. Furthermore, one
infers that $M_l(z)$ is in the generalized Nevanlinna class $N_{\kappa_l}^\infty$ with $\kappa_l=\floor{l/2 + 3/4}$.

For more information we refer to Section~4 of \cite{gz}, to \cite{ek}, where
the limit circle case $l\in[-1/2,1/2)$ is considered, and to Section~5 of \cite{fl_09},
where the Coulomb Hamiltonian $H_l - a/ x$ is worked out (see also \cite{DSh_00}, \cite{KurLug}).

\section{Asymptotics of solutions}\label{sec:ass_sol}

\subsection{General results}\label{sec:iii.1}

The main object of the following sections is the perturbed Bessel differential expression \eqref{eq-bes}.
In order to avoid cumbersome case distinctions we will exclude the special case $l=-\frac{1}{2}$ most of the time.
Since the operator is limit circle for $l\in[-\frac{1}{2},\frac{1}{2})$ this case is of no interest to us.

We begin with the following preliminary result (see \cite[Lemma~2.2]{kst}).

\begin{lemma}\label{lemphi}
Abbreviate $\ti{q}(x) = q(x)$ for $l>-\frac{1}{2}$ and $\ti{q}(x) = (1-\log(x))q(x)$ for $l=-\frac{1}{2}$.
Assume that $x \ti{q}(x)\in L^1(0,1)$.
Then there is a solution $\phi(z,x)$ of $\tau u = z u$ which is entire with respect to $z$ and satisfies the integral
equation
\be \label{a1}
\phi(z,x) = \phi_l(z,x) + \int_0^x G_l(z,x,y) q(y) \phi(z,y) dy,
\ee
where
\be \label{a2}
G_l(z,x,y) = \phi_l(z,x) \theta_l(z,y) - \phi_l(z,y) \theta_l(z,x)
\ee
is the Green function of the initial value problem. Moreover, this solution satisfies the estimate
\be\label{estphi}
| \phi(z,x) - \phi_l(z,x)| \leq C \left(\frac{x}{1+ |z|^{1/2} x}\right)^{l+1} \E^{|\im(z^{1/2})| x} \int_0^x \frac{y |\ti{q}(y)|}{1 +|z|^{1/2} y} dy.
\ee
The derivative is given by
\be  \label{a3}
\phi'(z,x) = \phi_l'(z,x) + \int_0^x \frac{\partial}{\partial x}G_l(z,x,y) q(y) \phi(z,y) dy
\ee
and satisfies the estimate
\be\label{estphi'}
| \phi'(z,x) - \phi_l'(z,x)| \leq C  \left(\frac{x}{1+ |z|^{1/2} x}\right)^l \E^{|\im(z^{1/2})| x} \int_0^x \frac{y |\ti{q}(y)|}{1 +|z|^{1/2} y} dy.
\ee
\end{lemma}

The next result plays a key role in the study of solutions $\phi(z,x)$ and $\theta(z,x)$.
\begin{lemma}\label{lem_ass}
Let $l>-1/2$ and $xq(x)\in L^p(0,1)$, $p\in[1,\infty]$, or
$l= -\frac{1}{2}$ and either $xq(x)\in L^p(0,1)$, $p\in(1,\infty]$, or $(1-\log(x)) x q(x) \in L^1(0,1)$ and $p=1$.

Then there exist two linearly independent solutions
$\phi(z,x)$ and $\theta(z,x)$ of $\tau u=z u$ such that
\be\label{eq-ass3}
 \phi(z,x)=x^{l+1}\widetilde{\phi}(z,x),\quad
 \theta(z,x)= \begin{cases} \frac{x^{-l}}{2l+1} \widetilde{\theta}(z,x), & l>-\frac{1}{2},\\
-\log(x) x^{1/2} \widetilde{\theta}(z,x), & l=-\frac{1}{2},\end{cases}
\ee
$\widetilde{\phi}(z,0)=\widetilde{\theta}(z,0)=1$, where
\be\label{eq-ass3B}
\widetilde{\phi}(z,.)\in W^{1,p}(0,1), \: p\in[1,\infty],
\quad \widetilde{\theta}(z,.)\in \begin{cases}
W^{1,p}(0,1), \: p\in [1,\infty], & 0< l,\\
W^{1,p}(0,1), \: p\in [1, \frac{-1}{2l}), & \frac{-1}{2} < l \leq 0,\\
C [0,1], & l=-\frac{1}{2}.\end{cases}
\ee
and, moreover, for $l>-1/2$,
\be\label{eq:31.09}
\lim_{x\to 0}x\widetilde{\phi}'(z,x)=\lim_{x\to 0}x\widetilde{\theta}'(z,x)=0,\quad \text{and}\quad \lim_{x\to 0}W_x(\theta(z),\phi(\zeta))= 1.
\ee
The functions $\phi(z,x)$ and $\theta(z,x)$ can be chosen entire with respect to $z$ and $\widetilde{\phi},\widetilde{\theta} \in C(\C\times[0,1])$.
Here $W^{1,p}(0,1)$ denotes the usual Sobolev space consisting of all
absolutely continuous functions whose derivative is in $L^p(0, 1)$.
\end{lemma}

Note: The restriction on $p$ in \eqref{eq-ass3B} in the case $\frac{-1}{2} < l \leq 0$ should be understood as $\widetilde{\theta}(z,.) \in W^{1,\ti{p}}$ for any
$\ti{p} \le \min(p,\frac{-1}{2l})$ since $L^{\ti{p}}(0,1) \subset L^p(0,1)$.

\begin{proof}
Without loss of generality we assume $z=0$ and we abbreviate $\hat{q}(x)=x q(x) \in L^p$.
We begin by making the ansatz
\[
\phi(x) =  x^{l+1} \E^{\int_0^x w(y) dy}
\]
such that $\phi$ solves $\tau \phi =0$ if and only if $w$ solves the Riccati equation
\[
w'(x) + w(x)^2 + \frac{2(l+1)}{x} w(x) = q(x).
\]
Now introduce (cf.\ Appendix~\ref{app:hi})
\[
(\mathcal{K}_\ell f)(x): = x^{-\ell-1} \int_0^x y^\ell f(y) dy
\]
and write
\[
w(x) = c(x) (\mathcal{K}_{2l+1}(c^{-1} \hat{q}))(x)
\]
for some continuous positive function $c$ to be determined. Then $w$ will satisfy our Riccati equation
if $c$ solves the integral equation
\[
c(x) = 1 - \int_0^x c(t)^2 (\mathcal{K}_{2l+1}(c^{-1} \hat{q}))(y) dy=:(Ac)(x).
\]
For $l> -1/2$, \eqref{hi_01a} implies
\[
Q(x) := \mathcal{K}_{2l+1}(|\hat{q}|)(x) \in L^p(0,a) \subset L^1(0,a)
\]
and we can choose $a$ so small that $L = 15\int_0^a Q(y) dy <1$.
Then, if we consider the ball $B_{1/2}(1)$ of radius $1/2$ around the constant function $1$
in $C[0,a]$ we obtain
\[
\|Af-1\|_\infty \le \int_0^a\|f\|^2_\infty \|f^{-1}\|_\infty Q(y)dy \le \int_0^a \frac{9}{4} 2 Q(y) dy < \frac{1}{2}, \quad f\in B_{1/2}(1).
\]
Similarly,
\begin{align*}
& \|Af-Ag\|_\infty=\Big\| \int_0^x \big( f(y)^2 (\mathcal{K}_{2l+1}(f^{-1} \hat{q}))(y) -g(y)^2 (\mathcal{K}_{2l+1}(g^{-1} \hat{q}))(y) \Big) dy \Big\|_\infty \leq\\
& \qquad \int_0^a \Big( 3 \|f-g\|_\infty 2 Q(y) + \frac{9}{4} 4\|f-g\|_\infty Q(y) \Big) dy \leq L \|f-g\|_\infty
\end{align*}
and thus we get existence of a solution $c\in B_{1/2}(1)$ by the contraction principle.
In summary, $w\in L^p(0,a)$ and $\phi(x) = x^{l+1} \ti{\phi}(x)$ with
\[
\ti{\phi}(x)= \E^{\int_0^x w(y) dy} \in W^{1,p}(0,a)
\]
as desired. To see that $\ti{\phi}'(x)=o(x^{-1})$, observe that $\ti{\phi}_l(z,x)$ has this property (cf.\ \eqref{eq:2.06}) and then use the estimate \eqref{estphi'}.
The case $l=-1/2$ is similar using \eqref{hi_log} instead of \eqref{hi_01a} in the case p=1.

A second solution of the required type follows from
\[
\hat{\theta}(x) = \phi(x) \int_x^c\frac{dy}{\phi^2(y)} = x^{-l} \ti{\phi}(x) (\hat{\mathcal{K}}_{2l+1}(\ti{\phi}^{-2}))(x),\quad
\hat{\mathcal{K}}_\ell(f):=x^\ell \int_x^c y^{-\ell-1} f(y) dy,
\]
by virtue of \eqref{hi_02B} and \eqref{hi_lim2}.

To see that $\phi(z,x)$ and $\theta(z,x)$ can be chosen entire, we note that $\phi(z,x)$ coincides with the entire solution from Lemma~\ref{lemphi}
up to a constant. Moreover, by \cite[Lemma~8.3]{kst2} there is a second entire solution $\theta(z,x) = \alpha(z) \hat{\theta}(z,x) + \beta(z) \phi(z,x)$. Since
$1 = W(\theta(z),\phi(z)) =  \alpha(z) W(\hat{\theta}(z),\phi(z)) = \alpha(z)$, we see that $\theta(z,x) =\hat{\theta}(z,x) + \beta(z) \phi(z)$
and since $\hat{\theta}(z,x) + \beta(z) \phi(z,x)$ has the same asymptotic properties near $x=0$, we are done.
\end{proof}

\begin{remark}{\ }
\begin{itemize}
\item
Clearly we have $\ti{\phi}(z,.), \ti{\theta}(z,.)\in AC_{\mathrm{loc}}^2(0,1)$ (see also Corollary \ref{cor_ass} below).
\item
The Coulomb case $q(x)= x^{-1}$ shows that for $l=0$ and $p=\infty$ the derivative of the solution $\theta(0,x)$ can have
a logarithmic singularity and thus is not bounded in general.
\item
The result shows that any operator associated with \eqref{eq-bes} and defined on $L^2(0,1)$ is nonoscillatory and thus is bounded from
below with purely discrete spectrum (cf.\ \cite[Thm.~2.4]{kst}).
\end{itemize}
\end{remark}

\begin{corollary}\label{cor_ass}
Let $l>-1/2$ and $xq(x)\in L^p(0,1)$, $p\in[1,\infty]$. The derivatives of the solutions from the previous lemma
have the form
\be\label{eq-ass3C}
 \phi'(z,x)=x^l\widehat{\phi}(z,x),\quad
 \theta'(z,x)= \frac{x^{-l-1}}{2l+1} \widehat{\theta}(z,x),\quad
 \widehat{\phi}(z,0)= l+1,\: \widehat{\theta}(z,0)=-l,
\ee
where
\be\label{eq-ass3D}
\widehat{\phi}(z,.)\in W^{1,p}(0,1), \: p\in[1,\infty],
\quad \widehat{\theta}(z,.)\in W^{1,p}(0,1), \: p\in \begin{cases}
[1,\infty], & 0< l,\\
[1, \frac{-1}{2l}), & \frac{-1}{2} < l \leq 0.\end{cases}
\ee
\end{corollary}

\begin{proof}
This follows from the previous lemma using $\widehat{\phi}(z,x)= (l+1)\widetilde{\phi}(z,x) + x \widetilde{\phi}'(z,x)$
together with the differential equation
\[
x \widetilde{\phi}''(z,x) = -2(l+1) \widetilde{\phi}'(z,x) + x (q(x) - z) \widetilde{\phi}(z,x)
\]
which implies $\widehat{\phi}'(z,x)= -l \ti{\phi}'(z,x) + x (q(x)-z) \ti{\phi}(z,x) \in L^p(0,1)$.

The calculation for $\theta$ is similar.
\end{proof}

\begin{remark}
Let us note that existence of a fundamental system of solutions satisfying \eqref{eq-ass3} and \eqref{eq-ass3C} was first established by B\^{o}cher \cite{boc}, see also \cite{lv}.
\end{remark}

\subsection{Series representation of $\phi(z,x)$}\label{sec:iii.2}

Lemma~\ref{lem_ass} provides the asymptotics of solutions at a singular endpoint $x=0$.
However, this information is insufficient for our needs. The main aim of this and the following subsections
is to prove Frobenius type representations for the entire solutions $\phi(z,x)$ and $\theta(z,x)$.
Throughout this section it will be convenient to abbreviate
\be
I_l= \begin{cases}
            [1,\infty],& 0<l,\\
            [1,\frac{-1}{2l}),&\frac{-1}{2}<l\le 0.
            \end{cases}
\ee

\begin{lemma}\label{lem_5.5}
Assume that $l> -1/2$ and $x q(x)\in L^p(0,1)$ for some $p\in I_l$.
Then the solution $\phi(z,x)$ admits the representation
\begin{align}\label{eq_5.14}
\phi(z,x) &= x^{l+1} \sum_{k=0}^\infty \frac{x^{2k} \widetilde{\phi}_k(z_0,x)}{k!}(z-z_0)^k,\\
\phi'(z,x) &= x^l \sum_{k=0}^\infty \frac{x^{2k} \widehat{\phi}_k(z_0,x)}{k!}(z-z_0)^k,
\end{align}
where
\be\label{eq_5.14C}
\widetilde{\phi}_k(z_0,.),\quad \widehat{\phi}_k(z_0,.)\in W^{1,p}(0,1),
\ee
with
\be
\widetilde{\phi}_k(z_0,0)=C^\phi_{l,k}, \quad
\widehat{\phi}_k(z_0,0)= (l+1+2k) C^\phi_{l,k},\quad C^\phi_{l,k}= \frac{(-1)^k}{4^k (l+3/2)_k},
\ee
and $(x)_j$ the Pochhammer symbol.

Moreover, for any $z_0\in\C$ and $k\in\N_0$
\be\label{eq_5.14D}
\lim_{x\to 0}x\widetilde{\phi}_k'(z_0,x)=0.
\ee
\end{lemma}

The proof of this lemma is based on the following result.

\begin{lemma}\label{lem_5.4B}
Let $l>- 1/2$ and $xq(x)\in L^p(0,1)$ for some $p\in I_l$.
Assume that $g_k(x) = x^{l+1+2k} \ti{g}_k(x)$ with $k>-1$ and
$\ti{g}_k \in W^{1,p}(0,1)$ (with $k>-\frac{1}{2}$ and $\ti{g}_k\in L^{\ti{p}}(0,1)$, $\ti{p}\in[1,\infty]$).
Then the solution of the following inhomogeneous problem
\[
(\tau -z) f_k= g_k,\qquad \lim_{x\to 0} x^{-(l+1)} f_k(x)=0,
\]
is given by
\be\label{5.4B}
f_k(x)=-\frac{x^{l+1+2(k+1)} \ti{f}_k(x)}{4(k+1)(l+k+3/2)},\qquad \ti{f}_k\in W^{1,p}(0,1)\: \big(\ti{f}_k \in L^{\ti{p}}(0,1)\big),
\ee
where $\ti{f}_k(0)= \ti{g}_k(0)$. Moreover,
\be\label{5.4C}
f_k'(x)=-\frac{x^{l+2(k+1)} \hat{f}_k(x)}{4(k+1)(l+k+3/2)},\qquad \hat{f}_k\in W^{1,p}(0,1)\: \big(\hat{f}_k \in L^{\ti{p}}(0,1)\big),
\ee
where $\hat{f}_k(0)= (l+1+2(k+1)) \ti{g}_k(0)$ and $\lim_{x\to 0} x \ti{f}_k'(x)=0$.
If, additionally, $\lim_{x\to 0} x \ti{g}_k'(x)=0$, then $\lim_{x\to 0} x \hat{f}_k'(x)=0$.
\end{lemma}

\begin{proof}
Observe that $f_k$ admits the representation
\[
f_k(x)=c_1\theta(z,x)+c_2\phi(z,x)+(\widetilde{G}_z g_k)(x),
\]
where
\begin{align}\nn
(\widetilde{G}_z g_k)(x) &= \theta(z,x)\int_0^x g_k(y)\phi(z,y) dy-\phi(z,x)\int_0^x g_k(y)\theta(z,y) dy\\  \label{eq_G}
&=\frac{x^{l+1+2(k+1)}}{(2l+1)}\left(\widetilde{\theta}(z,x)(\mathcal{K}_{2(l+k+1)}(\ti{g}_k \widetilde{\phi}))(x) - \widetilde{\phi}(z,x)(\mathcal{K}_{2k+1}(\ti{g}_k \widetilde{\theta}))(x)\right).
\end{align}
Since $\ti{g}_k, \widetilde{\phi}, \widetilde{\theta}\in W^{1,p}(0,1)$ we get $x^{-l-1-2(k+1)} (\widetilde{G}_z g_k)(x)\in W^{1,p}(0,1)$ by Lemma~\ref{lem:KlW}.
Similarly, if $\ti{g}_k\in L^p(0,1)$, then Lemma~\ref{lem:Kl} yields $x^{-l-1-2(k+1)} (\widetilde{G}_z g_k)(x)\in L^p(0,1)$. Moreover, the condition
$\lim_{x\to 0}x^{-(l+1)}f(x)=0$ implies $c_1=c_2=0$, that is, $f_k= \widetilde{G}_z g_k$.

Next, by \eqref{hi_01contin} we find
\begin{align*}
\ti{f}_k(0)= \lim_{x\to 0}x^{-(l+2k+3)}\widetilde{G}_z(g_k) =
\frac{1}{2l+1}\left(\frac{\ti{g}_k(0)\widetilde{\phi}(z,0)}{2l+2k+3}- \frac{\ti{g}_k(0) \widetilde{\theta}(z,0)}{2k+2}\right)\\
=-\frac{\ti{g}_k(0)}{4(k+1)(l+k+3/2)}.
\end{align*}
The claim about the derivatives follows using Corollary~\ref{cor_ass} and
\[
(\widetilde{G}_z g_k)'(x) = \frac{x^{l+2(k+1)}}{(2l+1)}\left(\widehat{\theta}(z,x)(\mathcal{K}_{2(l+k+1)}(\ti{g}_k \widetilde{\phi}))(x) -\widehat{\phi}(z,x)(\mathcal{K}_{2k+1}(\ti{g}_k \widetilde{\theta}))(x)\right).
\]
Finally, $\hat{f}_k(x) = (l+1+2(k+1)) \ti{f}_k(x) + x \ti{f}_k'(x)$ implies $\lim_{x\to 0} x \ti{f}_k'(x) =0$ and
if $\lim_{x\to 0} x \ti{g}_k'(x)=0$, then \eqref{hi_lim1} implies $\lim_{x\to 0} x \hat{f}_k'(x)=0$.
This completes the proof.
\end{proof}

\begin{proof}[Proof of Lemma~\ref{lem_5.5}]
Since $\phi(z,x)$ is entire in $z$, we get
\[
\phi(z,x)=\sum_{k=0}^\infty\frac{\phi^{(k)}(z_0,x)}{k!}(z-z_0)^k,\qquad \phi^{(k)}(z_0,x)=\frac{\partial^k}{\partial z^k}\phi(z,x)\Big|_{z=z_0}.
\]
Further, observe that $\phi^{(0)}(z_0,x)=\phi(z_0,x)$ and the derivative $\phi^{(k)}(z,x)$, $k\in \N_0$, satisfies the following equation
\[
(\tau - z) \phi^{(k+1)}(z,x) = (k+1) \phi^{(k)}(z,x).
\]
Moreover, by Lemma~\ref{lem_ass}, the solution $\phi(z,x)$ admits the representation
\[
\phi(z,x)=x^{l+1}\widetilde{\phi}(z,x),\quad \widetilde{\phi}(z,.)\in W^{1,p}(0,1), \quad \widetilde{\phi}\in C\big(\C,[0,1]\big).
\]
Due to the Cauchy integral formula, $\partial^k_z\widetilde{\phi} \in C\big(\C,[0,1]\big)$ and by $\widetilde{\phi}(z,0)\equiv 1$
we conclude that $\partial^k_z\widetilde{\phi}(z,0)= 0$, that is,
\[
\lim_{x\to 0} x^{-(l+1)}\phi_k(z_0,x)=0.
\]
Using Lemma~\ref{lem_5.4B} we obtain by induction
\[
\phi^{(k)}(z_0,x)=  x^{l+1+2k} \widetilde{\phi}_k(z_0,x), \quad \widetilde{\phi}_k(z_0,x)\in W^{1,p}(0,1),\ \ \widetilde{\phi}_k(z_0,0)=C^\phi_{l,k},
\]
which finishes the proof if $\ti{g}_k\in W^{1,p}$. The case $\ti{g}_k\in L^{\ti{p}}$ is similar.
\end{proof}

\subsection{Series representation of $\theta(z,x)$}\label{sec:iii.3}

The representation of the second solution is not unique since we can add $F(z)\phi(z,x)$, where $F$ is an arbitrary real entire function.
However, the singular part of $\theta(z,x)$ admits a Frobenius type decomposition.
Namely, the main result of this subsection is the following representation of $\theta(z,x)$.

\begin{lemma}\label{lem:theta_fr}
Let $l> -1/2$ and $xq(x)\in L^p(0,1)$for some $p\in I_l$.
Set $n_l:=\floor{l+1/2}$ and $\eps_l = n_l -l \in(-\frac{1}{2},\frac{1}{2}]$.
Then the solution $\theta(z,x)$ admits the following representation
\be\label{eq_5.14t}
\theta(z,x) = \frac{x^{-l}}{2l+1} \sum_{k=0}^\infty \frac{x^{2k}\widetilde{\theta}_k(z_0,x)}{k!}(z-z_0)^k + F(z) \phi(z,x),
\ee
with, if $l+1/2\not\in\N$,
\be\label{eq:3.15}
\widetilde{\theta}_k(z_0,.) \in W^{1,p}(0,1), \quad
\begin{cases} p\in I_l, & k<n_l,\\
p\in I_l, & k\ge n_l \text{ and } \eps_l\in(-\frac{1}{2},0) \text{ or } n_l=0,\\
p\in I_l \cap [1, \frac{1}{2\eps_l}), & k\ge n_l \text{ and } \eps_l \in[0,\frac{1}{2}),
\end{cases}
\ee
\be\label{eq:3.16}
\widetilde{\theta}_k(z_0,0) = C^\theta_{l,k} := \frac{(-1)^k}{4^k (-l+\frac{1}{2})_k},
\ee
and, if $l+1/2 = n_l\in\N$,
\be\label{eq:3.17}
\widetilde{\theta}_k(z_0,.) \in  \begin{cases}
W^{1,p}(0,1), \quad \widetilde{\theta}_k(z_0,0) = C^\theta_{l,k}, & k < n_l, \ p\in I_l,\\
\hat{C}^\theta_{l,k} \log(x) + W^{1,\ti{p}}(0,1), \: \ti{p}<p,  & k \ge n_l,\: p\in I_l \cap (1,\infty],\\
\hat{C}^\theta_{l,k} \log(x) (1 + o(1)), & k \ge n_l,\: p=1, \end{cases}
\ee
\be
C^\theta_{l,k} := \frac{1}{4^k (l-k+\frac{1}{2})_k}, \quad
\hat{C}^\theta_{l,k} := \frac{-C^\phi_{l,k-n_l}}{4^l \Gamma(l+1/2)},
\ee
and $F(z)$ is a real entire function and any polynomial part of degree up to order $n_l$ could be absorbed
in the series.

For the derivative we obtain
\be
\theta'(z,x) = \frac{x^{-l-1}}{2l+1} \sum_{k=0}^\infty \frac{x^{2k}\widehat{\theta}_k(z_0,x)}{k!}(z-z_0)^k + F(z) \phi'(z,x),
\ee
where $\widehat{\theta}_k$ is of the same nature as $\widetilde{\theta}_k$ with $C^\theta_{l,k}$, $\hat{C}^\theta_{l,k}$
replaced by $(-l+2k) C^\theta_{l,k}$, $(-l+2k) \hat{C}^\theta_{l,k}$, respectively.

Furthermore,
\be\label{eq:3.18}
\lim_{x\to 0}x\widetilde{\theta}_k'(z,x)=0, \quad \text{if} \quad k \in
\begin{cases}\N_0, &  l+1/2\notin\N,\\
k\le n_l-1, & l+1/2\in\N. \end{cases}
\ee
If $l+1/2\in\N$ and $k\ge n_l$, then
\be\label{eq:3.19}
\begin{cases}
\lim_{x\to 0}x \widetilde{\theta}_{k}'(z,x)= \hat{C}^\theta_{l,k}, & p\in (1,\infty],\\
\lim_{x\to 0} x^{1+\eps} \widetilde{\theta}_{k}'(z,x)= 0, \:\: \eps>0, & p=1, \: k \ge n_l.
\end{cases}
\ee
\end{lemma}

To prove this result we need again a preliminary lemma.

\begin{lemma}\label{lem_5.4}
Let $l>- 1/2$ and $xq(x)\in L^p(0,1)$ for some $p\in I_l$.
Assume that $g_k(x) = x^{-l+2k} \ti{g}_k(x)$ with $k\geq 0$ and
$\ti{g} \in W^{1,p}(0,1)$. Then the solution of the following inhomogeneous problem
\[
(\tau -z) f_k= g_k,\qquad \lim_{x\to 0} x^l f_k(x)=0,
\]
is given by the following formulas:

In the case $k<l - \frac{1}{2}$ we have
\be\label{eq:3.21}
f_k(x)=-\frac{x^{-l+2(k+1)} \ti{f}_k(x)}{4(k+1)(k-l+1/2)},
\ee
where
\be
\ti{f}_k\in W^{1,p}(0,1), \:
\begin{cases} p\in [1, \frac{1}{2(k-l+1)}), & l-1 \le k < l -\frac{1}{2},\\
p\in [1,\infty], & k<l-1, \end{cases} \quad
\ti{f}_k(0)= \ti{g}_k(0),
\ee
and $\lim_{x\to 0} x \ti{f}_k'(x)=0$.

In the case $k=l-\frac{1}{2}$ we have
\be
f_k(x)= \begin{cases}
-\frac{\ti{g}_k(0)\log(x)}{2l+1} \phi(z,x) + x^{l+1} \ti{f}_k(x),\quad \ti{f}_k\in W^{1,p}(0,1), &
p\in (1,\infty],\\
-\frac{\ti{g}_k(0)\log(x)}{2l+1} x^{l+1} + x^{l+1} \ti{f}_k(x),\quad \ti{f}_k = o(\log(x)), & p=1.
\end{cases}
\ee
If, additionally, $\lim_{x\to 0} x \ti{g}_k'(x)=0$, then $\lim_{x\to 0} x \ti{f}_k'(x)= 0$ for $p\in(1,\infty]$
and $\lim_{x\to 0} x^{1+\eps} \ti{f}_k'(x)=0$, $\eps>0$, for $p=1$.

For the derivative we obtain
\be\label{eq:3.21d}
f_k'(x)=-\frac{x^{-l-1+2(k+1)} \hat{f}_k(x)}{4(k+1)(k-l+1/2)},
\ee
where $\hat{f}_k$ is of the same type as $\ti{f}_k$ with $\hat{f}_k(0) = (-l+2(k+1)) \ti{g}_k(0)$.

If $\ti{g}_k\in L^{\ti{p}}(0,1)$ then \eqref{eq:3.21}, \eqref{eq:3.21d} hold with $\ti{f}_k, \hat{f}_k \in L^{\ti{p}}(0,1)$,
respectively.
\end{lemma}

\begin{proof}
Observe that $f_k$ admits the representation
\[
f_k(x)=c_1\theta(z,x)+c_2\phi(z,x)+(G_z g_k)(x),
\]
where
\begin{align*}
(G_z g_k)(x)&= \theta(z,x)\int_0^x g_k(y)\phi(z,y) dy + \phi(z,x)\int_x^1 g_k(y)\theta(z,y) dy\\
&=\frac{x^{-l+2(k+1)}}{(2l+1)}\left(\widetilde{\theta}(z,x)(\mathcal{K}_{2 k+1}(\ti{g}_k \widetilde{\phi}))(x) +
\widetilde{\phi}(z,x)(\hat{\mathcal{K}}_{2(l-k)-1}(\ti{g}_k \widetilde{\theta}))(x)\right).
\end{align*}
Since $\ti{g}_k, \widetilde{\phi}, \widetilde{\theta}\in W^{1,p}(0,1)$ we get $x^{l-2(k+1)} (G_z g_k)(x)\in W^{1,p}(0,1)$ in the case
$k<l-\frac{1}{2}$ by Lemma~\ref{lem:KlW}.
Moreover, the condition $\lim_{x\to 0}x^{-(l+1)}f(x)=0$ implies $c_1=0$, that is, $f_k= c_2 \phi + G_z g_k$.

Next, by \eqref{hi_01contin} and \eqref{hi_02contin} we find
\begin{align*}
\ti{f}_k(0)= \lim_{x\to 0}x^{l-2(k+1)} G_z(g_k) =
\frac{1}{2l+1}\left(\frac{\ti{g}_k(0)\widetilde{\phi}(z,0)}{2k+2} + \frac{\ti{g}_k(0) \widetilde{\theta}(z,0)}{2(l-k)-1}\right)\\
=-\frac{\ti{g}_k(0)}{4(k+1)(k-l+1/2)}.
\end{align*}
The rest follows as in Lemma~\ref{lem_5.5}.
This completes the proof in the case $k<l-\frac{1}{2}$.

In the case $k=l-\frac{1}{2}$ use \eqref{hi_03b} and \eqref{hi_03c},
which finishes the proof in the case $k=l-\frac{1}{2}$.
\end{proof}

Note that the case $k>l-\frac{1}{2}$ is covered by Lemma~\ref{lem_5.4B}.

\begin{proof}[Proof of Lemma~\ref{lem:theta_fr}]
The proof is similar to the proof of Lemma~\ref{lem_5.5}. First we have to use Lemma~\ref{lem_5.4} to obtain the
coefficients for $k\le n_l$ (in the case $k=n_l=0$ use Lemma~\ref{lem_5.4}). Then we note that by Lemma~\ref{lem_5.4B}
\[
\theta^{(n_l+1)}(z,x) = \check{\theta}^{n_l+1}(z,x) + G(z) \phi(z,x), \quad \lim_{x\to 0} x^{-l-1} \check{\theta}^{n_l+1}(z,x) =0.
\]
Hence, replacing $\theta(z,x) \to \theta(z,x) - F(z) \phi(z,x)$, where $F(z)$ is an entire function such that $F^{(n_l+1)}(z)= G(z)$,
we see that we can choose $G(z)=0$ without loss of generality. Thus we can assume
\[
\lim_{x\to 0} x^{-l-1} \theta^{(k)}(z_0,x) =0, \quad k>n_l,
\]
and continue to determine the coefficients for $k> n_l$ using Lemma~\ref{lem_5.4B}. For the case $l+1/2=n_l$ use \eqref{Kllog1}, \eqref{Kllog2}
together with the facts $x^{-1} (\ti{\phi}(z,x) -1) \in L^{\ti{p}}(0,1)$ and $\log(x) \ti{\phi}'(z,x) \in L^{\ti{p}}(0,1)$ for any $\ti{p}<p$ (recall that
functions in $W^{1,p}(0,1)$ are H\"older continuous with exponent $\gam=1-\frac{1}{p}$ for the first claim and H\"olders inequality for the second claim).

Concerning \eqref{eq:3.19} in the case $p>1$ observe that one can strengthen \eqref{eq:31.09} to read
$x\widetilde{\phi}'(z,x)=O(x^{1-1/p})$ and $x\widetilde{\theta}'(z,x)=O(x^{1-1/p})$.
\end{proof}

Lemma~\ref{lem_5.5} shows that the entire solution $\phi(z,x)$ is determined uniquely and has a Frobenius type form.
The solution $\theta(z,x)$ also has a Frobenius type form but it is not unique since we can add $F(z)\phi(z,x)$, where $F$ is an arbitrary real entire function.
Our next aim is to fix $F(z)$ in a suitable way.

\begin{definition}\label{def_thetafrob}
The solution $\theta(z,x)$ is called a Frobenius solution if
\be\label{eq:3.22}
F^{(n_l+1)}(z)=\lim_{x\to0}x^{-(l+1)}\frac{\partial^{(n_l+1)}}{\partial z^{(n_l+1)}}\theta(z,x)\equiv 0,
\ee
that is $\theta(z,x)$ is a Frobenius solution if and only if the function $F(z)$ in the representation \eqref{eq_5.14t} is a polynomial of order at most
$n_l:=\floor{l+1/2}$.
\end{definition}

\begin{remark}
There is another way to define a Frobenius solution:
Choose points $z_0,\dots, z_{n_l}$ and let
\[
L_j(z) = \prod_{k=0,k\ne j}^{n_l} \frac{z-z_k}{z_j-z_k}
\]
be the Lagrange interpolation polynomials. Then one can require that
\be
\lim_{x\to 0} W_x(\theta(z,x), \sum_{j=0}^{n_l} L_j(z) \theta(z_j,x)) = F(z) - \sum_{j=0}^{n_l} L_j(z) F(z_j)
\ee
vanishes. To see this just observe that
\[
\sum_{j=0}^{n_l} L_j(z) \theta(z_j,x) = \theta(z,x) + \left(F(z) - \sum_{j=0}^{n_l} L_j(z) F(z_j)\right) \phi(z,x) + o(x^{l+1}).
\]
\end{remark}

In particular, note that for a Frobenius solution we can choose $F(z)=0$ in \eqref{eq_5.14t} without loss of generality.

\begin{corollary}\label{cor:wronski}
Let $l> -1/2$ and $xq(x)\in L^1(0,1)$.
Let $\phi(z,x)$ and $\theta(z,x)$ be the solutions of $\tau u = z u$ constructed in Lemma~\ref{lem_ass}.
Then for any $z,\zeta\in\C$:
\begin{enumerate}
\item $W_x(\theta(z),\phi(z))\equiv 1$,
\item $\lim\limits_{x\to 0}W_x(\theta(z),\phi(\zeta))= 1$,
\item $\lim\limits_{x\to 0}W_x(\phi^{(i)}(z),\phi^{(j)}(\zeta))= 0$, $i,j\in\N_0$,
\item $\lim\limits_{x\to 0}W_x\big(\theta^{(i)}(z),\phi^{(j)}(\zeta)\big)=0$ if $i+j\ge 1$,
\item $\lim\limits_{x\to 0}W_x(\theta^{(i)}(z),\theta^{(j)}(\zeta))= \begin{cases} \infty, & i\ne j, i+j< n_l,\\
\infty, &  i\ne j, i+j=n_l, i j=0,\\
(j-i) C_{l,j}^\theta C_{l,i}^\theta, & i+j=n_l, i j >0, \end{cases}$
\item if $i+j\ge n_l+1$, then $\lim\limits_{x\to 0}W_x\big(\theta^{(i)}(z),\theta^{(j)}(\zeta)\big)=\begin{cases}-F^{(i)}(z),& j=0,\\
    F^{(j)}(\zeta), & i=0,\\
    0, & ij\neq 0.\end{cases}$
\end{enumerate}
\end{corollary}
\begin{proof}
$(i)$ and $(ii)$. This was already part of Lemma~\ref{lem_ass} (cf.\ \eqref{eq:31.09}).

$(iii)$. Observe that, by Lemma~\ref{lem_5.5}, $\phi^{(j)}(z,x)=x^{l+1+2j}\widetilde{\phi}_j(z,x)$, where $\widetilde{\phi}_j(z,.)\in W^{1,p}(0,1)$
satisfies \eqref{eq_5.14D}.

$(iv)$. By $(iii)$, we can assume without loss of generality that $\theta$ is of Frobenius type, i.e., $F\equiv 0$ in \eqref{eq_5.14t}.
Furthermore, by Lemma~\ref{lem:theta_fr}, $\theta^{(j)}(z,x)=x^{-l+2j}\widetilde{\theta}_j(z,x)$, where $\widetilde{\theta}_j(z,.)$ is given
by \eqref{eq:3.15} or \eqref{eq:3.17}. Taking into account \eqref{eq_5.14D} and \eqref{eq:3.18}, \eqref{eq:3.19} proves the claim.

$(v)$. Note that
\begin{align*}
W_x(x^{-l+2i}\widetilde{\theta}_i(z),x^{-l+2j}\widetilde{\theta}_j(z))&=2(j-i)x^{2(i+j-l)-1}\widetilde{\theta}_i(z,x)\widetilde{\theta}_j(z,x)\\
&\quad +x^{2(i+j-l)}W_x\big(\widetilde{\theta}_i(z),\widetilde{\theta}_j(z)\big).
\end{align*}
Since $2(i+j-l)\le 2(n_l-l)=2(\floor{l+1/2}-l)\le 1$, \eqref{eq:3.18} and \eqref{eq:3.19} complete the proof of $(v)$.

The proof of $(vi)$ follows from $(ii)$--$(iv)$ and the representation from Lemma~\ref{lem:theta_fr}.
\end{proof}
\begin{corollary}\label{cor:3.7}
Let $l> -1/2$ and $xq(x)\in L^1(0,1)$.
Let $\theta(z,x)$ be the solutions of $\tau u = z u$ constructed in Lemma~\ref{lem_ass}. Then $\theta(z,x)$ is a Frobenius type solution if and only if
\be\label{eq:theta_wro}
\lim_{x\to 0}W_x(\theta^{(n_l+1)}(z),\theta(\zeta))\equiv 0,\quad z,\zeta\in\C.
\ee
\end{corollary}
\begin{proof}
Combining Corollary \ref{cor:wronski} $(vi)$ with \eqref{eq:3.22}, we complete the proof.
\end{proof}

\section{Singular $m$-functions}
\label{sec:m-func}

\subsection{Some general facts}

Now let us look at perturbations
\be
H = H_l + q(x)
\ee
assuming that the potential $q$ satisfies the following conditions:

\begin{hypothesis}\label{hyp_q}
Let $l\in [-\frac{1}{2},\infty)$. Suppose $q\in L^1_{\mathrm{loc}}(\R_+)$ is real-valued such that
\be \label{f1}
\begin{cases} x\, q(x) \in L^1(0,1), & l>-\frac{1}{2},\\
x(1-\log(x)) q(x) \in L^1(0,1), & l = -\frac{1}{2}. \end{cases}
\ee
Moreover, assume that $\tau = \tau_l +q$ is limit point at $\infty$.
\end{hypothesis}

Under Hypothesis~\ref{hyp_q} the differential equation $H = H_l +q$ is limit circle at $x=0$ if $l\in[-\frac{1}{2},\frac{1}{2})$ and
limit point at $x=0$ for $l\geq \frac{1}{2}$.  In particular, $H$ associated with the boundary conditions at $x=0$
(for $l\in[-\frac{1}{2},\frac{1}{2})$)
\be
\lim_{x\to0} x^l ( (l+1)f(x) - x f'(x))=0
\ee
is self-adjoint by \cite[Thm.~2.4]{kst}. See also \cite{bg} for a characterization of all possible boundary conditions in terms of
Rellich's {\em Anfangszahlen}.

The results from the previous section also give us information on the associated scale of spaces.
We begin with characterizing the form domain of $H$.

\begin{lemma}
Suppose Hypothesis~\ref{hyp_q} holds. Assume additionally that $H$ is bounded from below.
The form domain of $H$ is given by
\be\label{eq:fdom}
\fdom(H) = \big\{ f \in L^2(\R_+) | f\in \mathrm{AC}(\R_+), \:  -f' + \frac{\phi'(\lam,.)}{\phi(\lam,.)} f \in L^2(\R_+)\big\}
\ee
for any $\lam$ below the spectrum of $H$. In particular, every $f\in \fdom(H)$ is of the form
\be
f(x) = x \ti{f}(x), \qquad \ti{f} \in L^2(0,1), \quad |\ti{f}(x)| \le \frac{const}{\sqrt{x}}.
\ee
\end{lemma}

\begin{proof}
Consider the operator $A = -\frac{d}{dx} + \phi'(\lam,.)/ \phi(\lam,.)$ which is a closed operator when defined on the domain given
on the right-hand side of \eqref{eq:fdom} (cf.~\cite[Problem~9.2]{tschroe}). Moreover, its adjoint is given by
$A^* = \frac{d}{dx} + \phi'(\lam,.)/ \phi(\lam,.)$ with domain
\begin{align*}
\dom(A^*) = \big\{ f \in L^2(\R_+) |& f\in \mathrm{AC}(\R_+), \:  f' + \frac{\phi'(\lam,.)}{\phi(\lam,.)} f \in L^2(\R_+),\\
&\lim_{x\to a,b} f(x) g(x) = 0, \: \forall g \in\dom(A)\big\}
\end{align*}
and hence one checks $H - \lam= A^* A$. In fact, the only nontrivial part is to identify the boundary condition at $a$ (if any).
However, since $\phi(\lam,.)$ is in the domain of $A^* A$ near $a$ by construction of $A$, equality of domain follows.
Consequently $\fdom(H)=\dom(A)$ finishing the first claim.

To prove the second claim let us consider the solution of the inhomogeneous equation
\be\label{eq:40.03}
-f'(x) + \frac{\phi'(\lam,x)}{\phi(\lam,x)} f(x) = g(x), \qquad g(x)\in L^2(0,1).
\ee
By Lemma~\ref{lem_ass}, $\phi(\lam,x)$ admits the representation
\[
\phi(\lam,x)=x^{l+1}\widetilde{\phi}(\lam,x),\quad \widetilde{\phi}(\lam,x)=\E^{\int_0^x w(y)dy},\quad w\in L^1(0,1).
\]
Therefore, $\phi'(\lam,x)/\phi(\lam,x)=\frac{l+1}{x} + w(x)$ and hence the solution of \eqref{eq:40.03} is given by
\[
f(x) = c_1 \phi(\lam,x) + x \ti{\phi}(\lam,x) \hat{\mathcal{K}}_l \Big(\frac{g}{\ti{\phi}(\lam)}\Big)(x),
\]
and \eqref{hi_02a} and \eqref{hi_sup1} complete the proof.
\end{proof}

Moreover, for the associated scale of spaces we obtain:

\begin{lemma} \label{lem:hrn}
Suppose Hypothesis~\ref{hyp_q} holds and $H$ is bounded from below.
Let $\hr_n$ be the scale of spaces associated with $H$ (cf.\ Appendix~\ref{app:ssp}). Then $f\in \hr_n$, $n\ge 0$, is of the form
\begin{align}
f(x) &= x^n \ti{f}(x), \qquad \ti{f} \in L^2(0,1), \quad 0 \le n \le \floor{l+1},\\
f'(x) &= x^{n-1} \hat{f}(x), \qquad \hat{f} \in L^2(0,1), \quad 2 \le n \le \floor{l+1}.
\end{align}
For $n \ge 1$ we have $|\ti{f}(x)| \le \frac{const}{\sqrt{x}}$ for $x\in(0,1)$ and for $n\ge 2$ we also have
$|\hat{f}(x)| \le \frac{const}{\sqrt{x}}$ for $x\in(0,1)$.

Moreover, any function of the form
\be
g(x) = x^{-n} \ti{g}(x), \qquad \ti{g}(x) \in L^2(0,1),\quad g(x) \in L^2(1,\infty),
\ee
lies in $\hr_{-n}$ for $ 0 \le n \le \floor{l+1}$.

If $H$ is not bounded from below the claim still holds for even $n$.
\end{lemma}

\begin{proof}
The first part follows from induction using Lemma~\ref{lem_5.4} (resp.\ Lemma~\ref{lem_5.4B}) starting from $\hr_0=L^2(\R_+)$
for the case of even $n$ and from $\hr_1=\fdom(H)$ for the case of odd $n$. The estimates for $\ti{f}$ and $\hat{f}$ follow
similarly using \eqref{hi_sup1}, \eqref{hi_sup2}.

To see the second part note that when $f_j(x) = x^n \ti{f}_j(x) \to f(x) = x^n \ti{f}(x)$ in $\hr_n$ then $\ti{f}_j(x) \to \ti{f}(x)$ in $L^2(0,1)$
and $f_j(x) \to f(x)$ in $L^2(1,\infty)$. The second claim is obvious and the first follows by inspection of the proof of Lemma~\ref{lem_5.4}
since the operators $\mathcal{K}_\ell$ and $\hat{\mathcal{K}}_\ell$ are continuous on $L^2(0,1)$. Hence it is easy to see that the linear
functional $f \mapsto \int_{\R_+} g(x) f(x) dx$ is continuous on $\hr_n$.
\end{proof}

Following \cite{gz}, we define $M(\cdot)$, the singular $m$-function for $\tau$, by
\begin{equation}\label{m-tau1}
\psi(z,x)=\theta(z,x)+M(z)\phi(z,x)\in L^2(1,+\infty),\quad z\in\C_+,
\end{equation}
where $\phi$ and $\theta$ are the entire solutions from Lemma~\ref{lem_ass}.

Let us recall some general facts from \cite{kst2}. First of all, associated with $M(z)$ is a spectral measure
$d\rho(\lam)$ and a unitary transform $U: L^2(\R_+) \to L^2(\R,d\rho)$ which maps $H$ to multiplication by the independent variable $\lam$. Both
$H$ and $U$ have unique extensions to the scale of spaces associated with $H$ (cf.\ Appendix~\ref{app:ssp}) which will be denoted by $\widetilde{H}$ and
$\widetilde{U}$, respectively. Moreover, recall
\be
(H-z)^{-1} f(x) = \int_0^\infty G(z,x,y) f(y) dy,
\ee
where
\be\label{defgf}
G(z,x,y) = \begin{cases} \phi(z,x) \psi(z,y), & y\ge x,\\
\phi(z,y) \psi(z,x), & y\le x,\end{cases}
\ee
is the Green function of $H$.

\begin{lemma}\label{lem:sing_pert}
Assume Hypothesis \ref{hyp_q} and let $\psi(z,.)$ be the Weyl solution defined by \eqref{m-tau1}. Abbreviate
\be
\psi^{(j)}(z,x) = \partial_z^j \psi(z,x),\qquad  j\in\N_0.
\ee
Then $\psi^{(j)}(z,.) \in \hr_{-n_l+2j} \setminus \hr_{-n_l+2j+1}$ with $n_l=\floor{l+\frac{1}{2}}$ and
\be\label{Upsij}
(\widetilde{U} \psi^{(j)}(z,.))(\lam) = \frac{j!}{(\lam-z)^{j+1}}, \qquad z\in\C\setminus\sig(H).
\ee
In particular,
\[
\psi^{(j)}(z,x)=  j! (\widetilde{H}-z)^{-j}\psi(z,x),\qquad j\in \N_0,
\]
and the distribution
\be\label{eq:phi_sing}
\varphi(x):= (\widetilde{H} -z)\psi(z,x)\in \hr_{-n_l-2}\setminus\hr_{-n_l-1}
\ee
does not depend on $z$, $(\widetilde{U} \varphi)(\lam)\equiv 1$.
\end{lemma}

\begin{proof}
We begin by observing that Lemma~\ref{lem_5.5} and Lemma~\ref{lem:theta_fr} imply
\be\label{asympsij}
\psi^{(j)}(z,x) = \begin{cases}
\frac{C^\theta_{l,j}}{2l+1} x^{-l+2j} (1+ o(1)), & j \ne l+1/2,\\
\frac{\hat{C}^\theta_{l,j}}{2l+1} x^{-l+2j} (\log(x)+ o(1)), &  j = l+1/2,
\end{cases}
\ee
for $j\le n_l$. Moreover, choosing $f(x)\in L^2(\R_+)$ with  compact support in $(0,1)$ we have
\[
\hat{\psi}(z,x) :=  (H-z)^{-1} f(x) = \left( \int_0^1 \phi(z,y) f(y) \right) \psi(z,x), \qquad x\ge 1.
\]
Since $\hat{\psi}(z,x)$ and all its $z$ derivatives are in $L^2(\R_+)$ we conclude that
$\psi^{(j)}(z,.) \in L^2(1,\infty)$. Thus Lemma~\ref{lem:hrn} shows  $\psi^{(j)}(z,.) \in \hr_{2(j-\gk_l)}\setminus \hr_{2(j-\gk_l+1)}$
for $j \le n_l$, where $\gk_l:= \floor{\frac{l}{2}+\frac{3}{4}}$.

Moreover, from \cite[Cor.~3.7]{kst2} we know
\be
(U \partial_z^j G(z,x,.))(\lam) = \frac{j! \phi(\lam,x)}{(\lam-z)^{j+1}},\qquad
(U \partial_z^j \partial_x G(z,x,.))(\lam) = \frac{j! \phi'(\lam,x)}{(\lam-z)^{j+1}}
\ee
for every $x\in(a,b)$, $k\in\N_0$, and every $z\in\C\setminus\sig(H)$. Hence for $j\ge \gk_l$ we obtain
\[
(U \ti{\psi}^{(j)}(z,x,.))(\lam) =  \frac{j!}{(\lam-z)^{j+1}} W_x(\theta(z),\phi(\lam)),
\]
where
\[
\ti{\psi}^{(j)}(z,x,y) = \begin{cases} \sum_{k=0}^j {j \choose k} \psi^{(j-k)}(z,y) W_x(\theta(z),\phi^{(k)}(z)), & y > x,\\
\sum_{k=0}^j {j \choose k} \phi^{(j-k)}(z,y) W_x(\theta(z),\psi^{(k)}(z)), & y < x.
\end{cases}
\]
Now \eqref{Upsij} for $j\ge \gk_l$ follows by letting $x\to 0$ using Corollary~\ref{cor:wronski}. To see it for $0\le j < \gk_l$
we will show $\psi^{(j-1)}(z) = \frac{1}{j}(\widetilde{H}-z) \psi^{(j)}(z)$ for $0<j\le \gk_l$. Choose $f\in \hr_{2(\gk_l-j)}$, then
\begin{align*}
\dpr[L^2]{(\widetilde{H}-z) \psi^{(j)}(z)}{f} &= \dpr[L^2]{\psi^{(j)}(z)}{ (H-z^*) f}\\
&= \lim_{\eps\downarrow 0} \int_\eps^\infty \psi^{(j)}(z,x) (\tau-z) f(x)^* dx\\
&= \lim_{\eps\downarrow 0} W_\eps(\psi^{(j)}(z),f^*) + \lim_{\eps\downarrow 0} \int_\eps^\infty j\psi^{(j-1)}(z,x) f(x)^* dx,
\end{align*}
where we have used integration by parts and $(\tau-z) \psi^{(j)}(z) = j\psi^{(j-1)}(z)$. Now alluding to \eqref{asympsij}
(and the corresponding statement for the $x$ derivative with $x^{-l+2j}$ replaced by $x^{-l+2j-1}$) and Lemma~\ref{lem:hrn}
we see that the Wronskian vanishes in the limit and that the second limit exists, that is,
\begin{align*}
\dpr[L^2]{(\widetilde{H}-z) \psi^{(j)}(z)}{f} &=  \int_0^\infty j\psi^{(j-1)}(z,x) f(x)^* dx,
\end{align*}
which shows \eqref{Upsij}.

Finally, to decide if $\psi^{(j)}(z,.) \in \hr_{2(j-\gk_l)}\setminus \hr_{2(j-\gk_l)+1}$ or $\psi^{(j)}(z,.) \in \hr_{2(j-\gk_l)+1}\setminus \hr_{2(j-\gk_l+1)}$
we consider the following integral
\[
I_{\gk_l}(z):=\int_0^\infty \psi^{(\gk_l)}(z,x)\ \psi^{(\gk_l-1)}(z,x)^* dx
\]
and recall $\psi^{(\gk_l-1)}(z,.) \in \hr_{-2} \setminus \hr_{0}$ and $\psi^{(\gk_l)}(z,.) \in \hr_{0} \setminus \hr_{2}$ plus
$\psi^{(j)}(z,.)\in L^2(1,+\infty)$ for all $j\in \N_0$. Moreover, by \eqref{asympsij}, we see that
$I_{\gk_l}(z)$ is finite if and only if $n_l=2\gk_l-1$. Since $\psi^{(\gk_l)}(z,x)=(\widetilde{H}-z)^{-1}\psi^{(\gk_l-1)}(z,x)$, the latter
means that $\psi^{(\gk_l-1)}(z,.) \in \hr_{-1} \setminus \hr_{0}$ if and only if $n_l=2\gk_l-1$. Otherwise, we get
$\psi^{(\gk_l-1)}(z,.) \in \hr_{-2} \setminus \hr_{-1}$. This completes the proof.
\end{proof}

\subsection{Main Theorem}

The main aim of this section is to show that the solution $\theta(z,x)$ can be chosen such that $M(z)$ belongs to the generalized
Nevanlinna class. To this end, let $\psi(\I,x)$ be the Weyl solution defined by \eqref{m-tau1} and introduce the function
\be\label{eq:def_mti}
\widetilde{M}(z)=\begin{cases}
\frac{(z^2+1)^{\gk_l}}{(\gk_l-1)!^2}\dpr[L^2]{\psi^{(\gk_l-1)}(\I)}{(\widetilde{H}-z)^{-1}\psi^{(\gk_l-1)}(\I)}, & n_l=2\gk_l-1,\\
\frac{(z^2+1)^{\gk_l}}{(\gk_l-1)!^2}\dpr[L^2]{\psi^{(\gk_l-1)}(\I)}{\big((\widetilde{H}-z)^{-1}-\mathcal{R}\big)\psi^{(\gk_l-1)}(\I)}, & n_l=2\gk_l,
\end{cases}
\ee
which is well defined for $z\in\C\setminus\sig(H)$ by Lemma~\ref{lem:sing_pert}.
Here $\mathcal{R}:=\re \big((\widetilde{H}-\I)^{-1}\big)=\frac{1}{2}\big((\widetilde{H}-\I)^{-1}+(\widetilde{H}+\I)^{-1}\big)$ and the inner product in \eqref{eq:def_mti} is understood as a pairing between $\hr_{-1}$ and $\hr_{1}$ or between $\hr_{-2}$ and $\hr_{2}$, respectively.
In the case $n_l=\gk_l=0$ (i.e., $l<\frac{1}{2}$), one has to set $\psi^{(-1)}(z) = \varphi = (\widetilde{H} -z)\psi(z)$.

Clearly, $\widetilde{M}(z)$ is a generalized Nevanlinna function and $\widetilde{M}\in N_{\gk_l}$. Moreover, by Lemma~\ref{lem:sing_pert}
\be\label{eq:spmes_mti}
\widetilde{M}(z)=\begin{cases}
\displaystyle (1+z^2)^{\gk_l}\int_\R \frac{1}{\lam - z} \frac{d\rho(\lam)}{(1+\lam^2)^{\gk_l}}, & n_l=2\gk_l-1,\\
\displaystyle (1+z^2)^{\gk_l}\int_\R \Big(\frac{1}{\lam - z}-\frac{\lam}{1+\lam^2}\Big)\frac{d\rho(\lam)}{(1+\lam^2)^{\gk_l}}, &n_l=2\gk_l,
\end{cases}
\ee
where $\rho$ is the spectral measure satisfying
\be\label{eq:spmes_mti_2a}
\int_\R\frac{d\rho(\lam)}{(1+|\lam|)^{2\gk_l+2}}<\infty.
\ee
Moreover, by Lemma \ref{lem:sing_pert} and \eqref{eq:phi_sing}, the representation \eqref{eq:def_mti} yields the following estimate for the measure
\be\label{eq:spmes_mti_2}
\int_\R\frac{d\rho(\lam)}{(1+|\lam|)^{n_l+2}}<\infty,\qquad \int_\R\frac{d\rho(\lam)}{(1+|\lam|)^{n_l+1}}=\infty.
\ee
Also, for the singular $m$-function $M(z)$ and the entire function $F(z)$ given in Lemma~\ref{lem_5.5} let us define the polynomials $P_M(z)$ and $P_F(z)$ of order at most $n_l$ by
\be\label{eq:pol_f}
P_f(z^*)=P_f(z)^*,\qquad P_f^{(j)}(\I)=f^{(j)}(\I),\quad j\in \{0,\dots, \floor{\frac{n_l}{2}}\},\quad f\in\{M,F\}.
\ee
With this notation our main result reads as follows:

\begin{theorem}\label{th:main}
Assume Hypothesis \ref{hyp_q}. Let the functions $M(z)$ and $\widetilde{M}(z)$ be defined by \eqref{m-tau1} and \eqref{eq:def_mti}, respectively. Then
\be\label{eq:m=ti_m}
M(z)=\widetilde{M}(z)+P_M(z)+G(z),\qquad z\in\C\setminus\rho(H),
\ee
where $G(z)=F(z)-P_F(z)$ and $F(z)$ is the entire function given in Lemma \ref{lem_5.5}.

The singular $m$-function $M(z)$ is a generalized Nevanlinna function from the class $N_{\gk_l}^\infty$ with $\gk_l=\floor{\frac{l}{2}+\frac{3}{4}}$ if and only if $F(z)$ is a real polynomial $\sum_{j=0}^m a_mz^m$ such that $m=2\gk_l+1$ and $a_m\ge 0$.
In particular, $M\in N_{\gk_l}^\infty$ if $\theta(z,x)$ is a Frobenius type solution.

The corresponding spectral measure $d\rho$ is given by \eqref{eq:spmes_mti} and satisfies \eqref{eq:spmes_mti_2}.
\end{theorem}

In order to prove this theorem we will distinguish the cases when $n_l$ is even and odd. Moreover, to make the proof more
transparent we will show the first cases $n_l=1$ and $n_l=2$ separately. The case $n_l=0$ already follows from \cite[Appendix~A]{kst2}
and will thus not be considered here.

\subsubsection{Step 1. The case $n_l=1$}

First, observe that $\gk_l=1$ since $l\in [1/2, 3/2)$.
Therefore, by  Lemma~\ref{lem:sing_pert}
\be\label{eq:41.01}
\psi(z,.)\in \hr_{-1}\setminus\hr_0,\qquad z\in\C\setminus\sig(H).
\ee

The latter enables us to introduce the function $\widetilde{M}(z)$ by \eqref{eq:def_mti}. Thus we get
\be\label{sp_mfunct}
\widetilde{M}(z)=(z^2+1) \dpr[L^2]{\psi(\I)}{(\widetilde{H}-z)^{-1}\psi(\I)}.
\ee

\begin{lemma}\label{lem_m=m}
Let $l\in [1/2, 3/2)$ and assume Hypothesis \ref{hyp_q}. Let the functions $M$ and $\widetilde{M}$ be defined by \eqref{m-tau1} and \eqref{sp_mfunct}, respectively. Then
\be\label{eq_connection_l=1}
M(z)=\widetilde{M}(z) +\im M(\I)\cdot z+\re M(\I)+ G(z), \qquad z\in\C\setminus\sig(H),
\ee
where the function $G$ is entire. Moreover, $G(z)=F(z)-\im F(\I)\cdot z-\re F(\I)$, where the function $F(z)$ is given by \eqref{eq_5.14t}.
\end{lemma}
\begin{proof}
Consider the following function for $x>0$
\be\label{eq:41.07}
Q_1(z,x):=(z^2+1)\int_x^{+\infty}\bigl((\widetilde{H}-z)^{-1}\psi(\I,t)\bigr)\psi(\I,t)^*dt.
\ee
Note that, the definition of $Q(z,x)$ is correct and
\be\label{eq:41.08}
\lim_{x\to 0}Q_1(z,x)=\widetilde{M}(z).
\ee
Furthermore, by Lemma \ref{lem:sing_pert} \[
\psi(z,x)-\psi(\I,x)=(z-\I)(\widetilde{H}-z)^{-1}\psi(\I,x),
\]
 and hence we get
\begin{align*}
Q_1(z,x) &=(z+\I)\int_x^{+\infty}(\psi(z,t)-\psi(\I,t))\psi(\I,t)^*dt= \\
&=(z+\I)\int_x^{+\infty}\psi(z,t)\psi(\I,t)^*dt-(z+\I)\int_x^{+\infty}\psi(\I,t)\psi(\I,t)^*dt\\
&=-W_x(\psi(z),\psi(\I)^*)+\frac{z+\I}{2\I}W_x(\psi(\I),\psi(\I)^*)\\
&=-W_x\Big(\psi(z)-\frac{z+\I}{2\I} \psi(\I)\ ,\ \psi(-\I)\Big).
\end{align*}
Therefore, by \eqref{eq:41.08}
\be\label{eq:41.09}
\widetilde{M}(z)=-\lim_{x\to 0}W_x\Big(\psi(z)-\frac{z+\I}{2\I} \psi(\I)\ ,\ \psi(-\I)\Big).
\ee
Using the definition \eqref{m-tau1} of $\psi(z,x)$, we obtain
\begin{align*}
&W_x\big(\psi(z),\psi(-\I)\big)
 =W_x\big(\theta(z),\theta(-\I)\big)+M(z)M(-\I)W_x\big(\phi(z),\phi(-\I)\big)\\
&\qquad +M(z)W_x\big(\phi(z),\theta(-\I)\big)+ M(-\I)W_x\big(\theta(z),\phi(-\I)\big).
\end{align*}
Combining \eqref{eq:41.09} with the last equality and using Corollary \ref{cor:wronski} (ii)--(iii), we finally get
\begin{eqnarray}\label{eq_5.11}
&\widetilde{M}(z)= \lim_{x\to 0}Q_1(z,x)
= M(z)-\im M(\I)\cdot z -\re M(\I)  \\
&- \lim_{x\to 0}W_x\Big(\theta(z)-\frac{z+\I}{2\I}\theta(\I)\ ,\ \theta(-\I)\Big).\nn
\end{eqnarray}
Further, setting $z_0=-\I$ in \eqref{eq_5.14t} we get the following representation of $\theta(z,x)$,
\begin{align*}
\theta(z,x)&=\theta_0(-\I,x)+\theta_1(-\I,x)(z+\I)+\frac{\theta_2(z,x)}{2}(z+\I)^2+F(z)\phi(z,x),\\
 &\quad \theta_j(z,x)=x^{-l+2j}\widetilde{\theta}_j(z,x),\qquad j\in\{0,1,2\}.
\end{align*}
Using this representation and noting that $\lim_{x\to 0}W_x(\theta_2(z),\theta(-\I))=0$, we see that the limit in \eqref{eq_5.11} exists and is an entire function in $z$.
Therefore, setting
\be\label{eq:4.6}
G(z) :=-\lim_{x\to 0}W_x\Big(\theta(z)-\frac{z+\I}{2\I}\theta(\I)\ ,\ \theta(-\I)\Big)=F(z)-\im F(\I) \cdot z-\re F(\I),
\ee
we have proven the claim.
\end{proof}

\begin{proof}[Proof of Theorem \ref{th:main} in the case $n_l=1$]
The first part is contained in Lemma \ref{lem_m=m}.

Further, combining \eqref{eq:spmes_mti}, \eqref{eq:spmes_mti_2}, \eqref{eq_connection_l=1}, and \eqref{eq:4.6}, by Theorem \ref{th:gnf} we see that $M\in N_1^\infty$ if and only if the function $G$ defined by \eqref{eq:4.6}, and hence the function $F$, is a polynomial satisfying \eqref{eq:b.08} with $\kappa=1$.

By Definition~\ref{def_thetafrob},  $\theta(z,x)$ is a Frobenius type solution if $F$ is a linear function and hence in this case $M\in N_1^\infty$.
\end{proof}

\subsubsection{Step 2. The case $n_l=2$}

Since $l\in [3/2,5/2)$ we get $\gk_l=1$.
Furthermore, by  Lemma~\ref{lem:sing_pert},
\[
\psi(z,x)\in \mathfrak{H}_{-2}\setminus\mathfrak{H}_{-1},\quad z\in\C\setminus\sig(H),
\]
and in this case \eqref{eq:def_mti} takes the form
\be\label{sp_mfunct_l=2}
\widetilde{M}(z)=(z^2+1) \dpr[L^2]{\psi(\I)}{\bigl((\widetilde{H}-z)^{-1}-\mathcal{R}\bigr)\psi(\I)},
\ee
Note also that $\widetilde{M}(z)\in N_{1}^\infty$ (cf. \eqref{eq:spmes_mti}, \eqref{eq:spmes_mti_2} and Theorem \ref{th:gnf}).

\begin{lemma}\label{lem_m=m_l=2}
Let $l\in [3/2, 5/2)$ and assume Hypothesis \ref{hyp_q}. Let the functions $M$ and $\widetilde{M}$ be defined by \eqref{m-tau1}
and \eqref{sp_mfunct_l=2}, respectively. Then
\[
M(z)=\widetilde{M}(z)+ \im M(\I)\cdot z + \re M(\I) +\frac{\im \dot{M}(\I)}{2}(z^2+1)+ G(z), \quad z\in \C\setminus\sig(H),
\]
where the function $G$ is entire. Moreover,
\[
G(z)=F(z)-\im F(\I)\cdot z-\re F(\I)-\frac{z^2+1}{2}\im \dot{F}(\I),
\]
 where $F(z)$ is given in Lemma~\ref{lem:theta_fr}.
\end{lemma}
\begin{proof}
For $x>0$, consider the function
\[
Q_2(z,x):=(z^2+1)\int_x^{+\infty}\bigl(((\widetilde{H}-z)^{-1}-\mathcal{R})\psi(\I,t)\bigr)\psi(\I,t)^*dt.
\]
Note that the definition of $Q_2(z,x)$ is correct and, moreover,
\be\label{eq:42.03}
\lim_{x\to 0}Q_2(z,x)=\widetilde{M}(z).
\ee
Furthermore,
\begin{align}
Q_2(z,x)=Q_1(z,x) -(z^2+1)\mathcal{R}_2(x),\nn
\end{align}
where $Q_1(z,x)$ is given by \eqref{eq:41.07} and
\begin{align*}
\mathcal{R}_2(x):= \int_x^{+\infty}\bigl(\mathcal{R}\psi(\I,t)\bigr)\psi(\I,t)^*dt
= \re \int_x^{+\infty}\bigl((\widetilde{H}-\I)^{-1}\psi(\I,t)\bigr)\psi(\I,t)^*dt \\
=\re \int_x^{+\infty}\dot{\psi}(\I,t)\psi(\I,t)^*dt
=-\re\ W_x\Big(\frac{\dot{\psi}(\I)}{2\I}+\frac{\psi(\I)}{4},\psi(-\I)\Big)\\
=-\re\ W_x\Big(\frac{\dot{\psi}(\I)}{2\I},\psi(-\I)\Big)=-\frac{1}{2}\ \im\ W_x\big(\dot{\psi}(\I),\psi(-\I)\big).
\end{align*}
Noting that
\begin{align*}
&W_x\big(\dot{\psi}(\I),\psi(-\I)\big)\\
&=W_x(\dot{\theta}(\I),\theta(-\I))+M(\I)W_x(\dot{\phi}(\I),\theta(-\I))
+\dot{M}(\I)W_x(\phi(\I),\theta(-\I))\\
&+M(-\I)W_x(\dot{\theta}(\I),\phi(-\I))+|M(\I)|^2 W_x(\dot{\phi}(\I),\phi(-\I))+\dot{M}(\I)M(-\I)W_x(\phi(\I),\phi(-\I)),
\end{align*}
and using Corollary \ref{cor:wronski}, we get
\begin{align}
\widetilde{M}(z)= & \lim_{x\to 0} Q_2(z,x)= \lim_{x\to 0} \Big(Q_1(z,x)+\frac{(z^2+1)}{2}\im \ W_x\big(\dot{\psi}(\I),\psi(-\I)\big)\Big)\\
= & M(z) - \im M(\I)\cdot z - \re M(\I) -\frac{\im \dot{M}(\I)}{2}(z^2+1) \nonumber\\
&-\lim_{x\to 0}W_x\Big(\theta(z)-\frac{z+\I}{2\I}\theta(\I)\ -\frac{z^2+1}{2}\im\big( \dot{\theta}(\I)\big)\ ,\ \theta(-\I)\Big).\nonumber
\end{align}
Noting that $\theta_k(\I,x)=\theta_k(-\I,x)^*$ and using \eqref{eq_5.14t} with $z_0=-\I$, we obtain
\begin{align*}
\im \big(\dot{\theta}(\I,x)\big)&=\frac{\dot{\theta}(\I,x)-\dot{\theta}(-\I,x)}{2\I}=\frac{\dot{\theta}(\I,x)-\theta_1(-\I,x)-\big(F(-\I)\phi(-\I,x)\big)^{\cdot}}{2\I}\\
&=\theta_2(-\I,x)-2\theta_3(\I,x)+\im \left(\big(F(\I)\phi(\I,x)\big)^\cdot\right).
\end{align*}
Finally, using Corollary \ref{cor:wronski}, we get after a straightforward calculation
\begin{align*}
\widetilde{M}(z)
&=M(z) - \im M(\I)\cdot z - \re M(\I) -\frac{\im \dot{M}(\I)}{2}(z^2+1) \\&+F(z)-\im F(\I)\cdot z-\re F(\I)-\frac{z^2+1}{2}\im \dot{M}(\I).
\end{align*}
\end{proof}

\begin{proof}[Proof of Theorem \ref{th:main} in the case $n_l=2$]
The first part is contained in Lemma \ref{lem_m=m}.

Further, using Lemma~\ref{lem_m=m_l=2} and \eqref{eq:spmes_mti}--\eqref{eq:spmes_mti_2}, by Theorem~\ref{th:gnf} we see that $M(z)$
is an $N_{1}$-function if and only if $F(z)$ is a polynomial satisfying \eqref{eq:b.08} with $\kappa=1$. In particular, $M\in N_1$ if
$\theta(z,x)$ is a Frobenius type solution.
\end{proof}

\subsubsection{Step 3. The case $n_l=2k+1$, $k\in\N$}

Assume that $l\in[2k+1/2,2k+3/2)$ for some fixed $k\in \N$. Note that in this case $\gk_l=\floor{\frac{l}{2}+\frac{3}{4}}=k+1$.
Furthermore, by Lemma~\ref{lem:sing_pert},
\[
\psi(z,x)\in \hr_{-2k-1}\setminus \hr_{-2k},\quad z\in\C\setminus\sig(H),
\]
and
\be\label{eq:43.01}
\psi^{(j)}(\I,x):=\partial^j_z\psi(\I,x)=j!(\widetilde{H}-\I)^{-j}\psi(\I,x) \in \hr_{-2k-1+2j}\setminus \hr_{-2k+2j}.
\ee
In this case, \eqref{eq:def_mti} takes the form
\be\label{eq:43.03}
\widetilde{M}(z)=\frac{(z^2+1)^{k+1}}{k!^2}\dpr[L^2]{\psi^{(k)}(\I)}{(\widetilde{H}-z)^{-1}\psi^{(k)}(\I)}.
\ee
Observe that $\widetilde{M}(z)$ is analytic in $\C\setminus\sig(H)$ and, moreover, $\widetilde{M}\in N^\infty_{\gk_l}$.

To proceed further we need the following formula
\be\label{eq:43.04}
(z-\I)^{k+1}(\widetilde{H}-z)^{-1}(\widetilde{H}-\I)^{-k}\psi(\I,x)=\psi(z,x)-\sum_{j=0}^{k}(z-\I)^j(\widetilde{H}-\I)^{-j}\psi(\I,x)
\ee
Combining \eqref{eq:43.04} with \eqref{eq:43.01}, we can rewrite \eqref{eq:43.03} as follows
\[
\widetilde{M}(z)=(z+\I)^{k+1}\Big(\frac{\psi^{(k)}(\I)}{k!},\ \psi(z)-\sum_{j=0}^k \frac{\psi^{(j)}(\I)}{j!}(z-\I)^{j} \Big).
\]
As in the previous subsections, we set
\be\label{eq:43.09}
Q_{n_l}(z,x):=\frac{(z+\I)^{k+1}}{k!}\int_x^\infty \Big(\psi(z,x)-\sum_{j=0}^k \frac{\psi^{(j)}(\I,t)}{j!}(z-\I)^{j}\Big)\psi^{(k)}(\I,t)^*dt.
\ee
Note that
\be\label{eq:43.03B}
\widetilde{M}(z)=\lim_{x\to 0}Q_{n_l}(z,x).
\ee
Furthermore, for $j\in\{0,\dots,k\}$ consider the following functions
\be\label{eq:43.10}
Q_{n_l,j}(z,x)=\frac{1}{k!}\int_x^\infty \psi^{(j)}(z,t)\ \psi^{(k)}(\I,t)^*dt,\quad Q_{n_l,j}(z,x)=\partial^j_z Q_{n_l,0}(z,x).
\ee
Thus we get
\be\label{eq:43.11}
Q_{n_l}(z,x)=(z+\I)^{k+1}\Big(Q_{n_l,0}(z,x)-\sum_{j=0}^k\frac{Q_{n_l,j}(\I,x)}{j!}(z-\I)^j\Big).
\ee
We begin with the function $Q_{n_l,0}(z,x)$. Clearly, we get
\begin{align}\nn
&Q_{n_l,0}(z,x)=\frac{1}{k!}\int_x^\infty \psi(z,t)\ \psi^{(k)}(\I,t)^*dt\\ \nn
&=\frac{1}{k!}\frac{\partial^k}{\partial \zeta^k}\int_x^\infty \psi(z,t)\ \psi(\zeta,t)dt\Big|_{\zeta=-\I}
=-\frac{1}{k!}\frac{\partial^k}{\partial\zeta^k}W_x\Big(\psi(z),\frac{\psi(\zeta)}{z-\zeta}\Big)\Big|_{\zeta=-\I}\\
&\qquad\qquad=-W_x\Big(\frac{\psi(z)}{(z+\I)^{k+1}}\ , \ \sum_{j=0}^k\frac{\psi^{(j)}(-\I)}{j!}(z+\I)^{j}\Big).\label{eq:43.16}
\end{align}
Moreover, using Corollary \ref{cor:wronski}, we obtain
\begin{align*}
&\lim_{x\to0}W_x\Big(\frac{M(z)\phi(z)}{(z+\I)^{k+1}}\ , \ \sum_{j=0}^k\frac{\theta^{(j)}(-\I)}{j!}(z+\I)^{j}\Big)=-\frac{M(z)}{(z+\I)^{k+1}},\\
& \lim_{x\to 0}W_x\Big(\frac{\theta(z)}{(z+\I)^{k+1}}\ , \ \sum_{j=0}^k\frac{(M\phi)^{(j)}(-\I)}{j!}(z+\I)^{j}\Big)=\sum_{j=0}^k\frac{M^{(j)}(-\I)}{j!}(z+\I)^{j-(k+1)},\\
&\lim_{x\to 0}W_x\Big(\frac{M(z)\phi(z)}{(z+\I)^{k+1}}\ , \ \sum_{j=0}^k\frac{(M\phi)^{(j)}(-\I)}{j!}(z+\I)^{j}\Big)=0.
\end{align*}
Setting
\be\label{eq:43.05}
\Theta(z,x):=W_x\Big(\frac{\theta(z)}{(z+\I)^{k+1}}\ , \ \sum_{j=0}^k\frac{\theta^{(j)}(-\I)}{j!}(z+\I)^{j}\Big)
\ee
and using \eqref{m-tau1}, by Corollary \ref{cor:wronski} we obtain
\be\label{eq:43.06}
\lim_{x\to0}\left(Q_{n_l,0}(z,x)-\Theta(z,x)\right)=\frac{1}{(z+\I)^{k+1}}\Big(M(z)-\sum_{j=0}^k \frac{M^{(j)}(-\I)}{j!}(z+\I)^{j}\Big).
\ee
On the other hand, by Lemma~\ref{lem:theta_fr} we get
\begin{align*}
\Theta(z,x)&=W_x\Big(\frac{\theta_F(z)+F(z)\phi(z)}{(z+\I)^{k+1}}\ , \ \sum_{j=0}^k\frac{\theta^{(j)}_F(-\I)+(F\phi)^{(j)}(-\I)}{j!}(z+\I)^{j}\Big)\\
&=W_x\Big(\sum_{i=0}^\infty\frac{x^{-l+2i}\widetilde{\theta}_i(-\I)}{i!}(z+\I)^{i-(k+1)}\ , \ \sum_{j=0}^k\frac{x^{-l+2j}\widetilde{\theta}_j(-\I)}{j!}(z+\I)^{j}\Big)\\
&\qquad-\frac{1}{(z+\I)^{k+1}}\Big(F(z)-\sum_{j=0}^k\frac{F^{(j)}(-\I)}{j!}(z+\I)^j\Big)\\
&=W_x\Big(\sum_{i=k+1}^\infty\frac{x^{-l+2i}\widetilde{\theta}_i(-\I)}{i!}(z+\I)^{i-(k+1)}\ , \ \sum_{j=0}^k\frac{x^{-l+2j}\widetilde{\theta}_j(-\I)}{j!}(z+\I)^{j}\Big)
-\widetilde{F}(z)\\
&\quad=\sum_{i,j:i+j\le k}\frac{(z+\I)^{i+j}}{(i+k+1)!j!}W_x\Big(\theta^{(i+k+1)}_F(-\I),\theta^{(j)}_F(-\I)\Big)-\widetilde{F}(z),
\end{align*}
where
\[
\widetilde{F}(z)=\frac{1}{(z+\I)^{k+1}}\Big(F(z)-\sum_{j=0}^k\frac{F^{(j)}(-\I)}{j!}(z+\I)^j\Big).
\]
Let us denote
\be\label{eq:43.15}
\Theta_F(z,x):=\sum_{i,j:i+j\le k}\frac{(z+\I)^{i+j}}{(i+k+1)!j!}W_x\Big(\theta^{(i+k+1)}_F(-\I),\theta^{(j)}_F(-\I)\Big).
\ee
Note that $\Theta_F(z,x)$ is a polynomial in $z$ of order at most $k$. Therefore,
\begin{align}
\Theta(z,x)-\sum_{j=0}^k\frac{\Theta^{(j)}(\I,x)}{j!}(z-\I)^k=\Theta_F(z,x)-\sum_{j=0}^k\frac{\Theta_F^{(j)}(\I,x)}{j!}(z-\I)^k\nn\\
-\widetilde{F}(z)+\sum_{j=0}^{k}\frac{\widetilde{F}^{(j)}(\I)}{j!}(z-\I)^j=-\widetilde{F}(z)+\sum_{j=0}^{k}\frac{\widetilde{F}^{(j)}(\I)}{j!}(z-\I)^j.\label{eq:43.14}
\end{align}
Noting that $Q_{n_l,j}(z,x)=\partial^j_z Q_{n_l,0}(z,x)$, by \eqref{eq:43.06} and \eqref{eq:43.16}, we get
\begin{align}\label{eq:43.07}
&\lim_{x\to0}\left(Q_{n_l,j}(z,x)-\frac{\partial^j}{\partial z^j} \Theta(\I,x)\right)=M_j(\I),\\
&M_j(\I)=\frac{\partial^j}{\partial z^j}\left(\frac{M(z)}{(z+\I)^{k+1}}-\frac{1}{(z+\I)^{k+1}}\sum_{j=0}^k \frac{M^{(j)}(-\I)}{j!}(z+\I)^{j}\right)\Big|_{z=\I}.\label{eq:43.08}
\end{align}
Observe that for arbitrary real entire function $f(z)$ the function
\[
\widetilde{f}(z)=(z+\I)^{-(k+1)}\Big(f(z)-\sum_{j=0}^k\frac{f^{(j)}(-\I)}{j!}(z+\I)^j\Big)
\]
is real and entire. Furthermore, the function
\[
(z+\I)^{k+1}\Big(\widetilde{f}(z)-\sum_{j=0}^k\frac{\widetilde{f}^{(j)}(-\I)}{j!}(z+\I)^j\Big)=f(z)-P_f(z)
\]
is also real and entire. Here $P_f$ is a real polynomial of order at most $2k+1$. Moreover, since $z=\pm \I$ is a zero of order at least $k+1$, the polynomial $P_f$ satisfies
\be\label{eq:43.17}
P_f(z^*)=P_f(z)^*,\qquad P^{(j)}_f(\I)=f^{(j)}(\I),\quad j\in\{0,\dots, k\}.
\ee

Combining \eqref{eq:43.03B} with \eqref{eq:43.10}, \eqref{eq:43.11},  \eqref{eq:43.06}, and \eqref{eq:43.14}--\eqref{eq:43.08} we finally get
\begin{align*}
\widetilde{M}(z)&=\lim_{x\to 0}(z+\I)^{k+1}\Big(Q_{n_l,0}(z,x)-\sum_{j=0}^k\frac{Q_{n_l,j}(\I,x)}{j!}(z-\I)^j\Big)\\
&=\lim_{x\to 0}(z+\I)^{k+1}\Big(\Theta(z,x)-\sum_{j=0}^k\frac{\Theta^{(j)}(\I,x)}{j!}(z-\I)^j\Big)\\
&\qquad+(z+\I)^{k+1}\Big(M_0(z)-\sum_{j=0}^k\frac{M_j(\I)}{j!}(z-\I)^j\Big)\\
&\qquad=M(z)-P_M(z) -(F(z)-P_F(z)),
\end{align*}
where $P_f(z)$ is a polynomial of order $n_l=2k+1$ satisfying \eqref{eq:43.17}.
Thus, we proved the following result.
\begin{lemma}\label{lem_m=m_2k+1}
Let $l\in[2k+1/2,2k+3/2)$ and assume Hypothesis \ref{hyp_q}. Let the functions $M$ and $\widetilde{M}$ be defined by \eqref{m-tau1} and \eqref{eq:43.03}, respectively. Then
\be
M(z)=\widetilde{M}(z)+P_M(z)+G(z),\quad z\in \C\setminus\sig(H),
\ee
where the function $G$ is entire. Moreover, $G(z)=F(z)-P_F(z)$, where the function $F$ is given in Lemma~\ref{lem:theta_fr} and $P_f(z)$ is a polynomial of order $n_l=2k+1$ satisfying \eqref{eq:43.17}.
\end{lemma}

\begin{proof}[Proof of Theorem \ref{th:main} in the case $n_l=2k+1$, $k\in\N$]
By Lemma~\ref{lem_m=m_2k+1}
and \eqref{eq:b.08},  $M(z)$ is an $N_{\gk_l}^\infty$-function if and only if $F(z)=\sum_{j=0}^m a_mz^m$ with either $m\le 2\gk_l$
or $m=2\gk_l+1$ with $a_m>0$.

In particular, $M(z)\in N_{\gk_l}$ if $\theta(z,x)$ is a Frobenius type solution.
\end{proof}

\subsubsection{Step 4. The case $n_l=2k+2$, $k\in\N$}

Finally, assume that $l\in[2k+3/2,2k+5/2)$ for some fixed $k\in \N$. Note that in this case $\gk_l=\floor{\frac{l}{2}+\frac{3}{2}}=k+1$. By
 Lemma~\ref{lem:sing_pert},
\[
\psi(z,.)\in \hr_{-2(k+1)}\setminus \hr_{-2k-1},\quad z\in\C\setminus\sig(H)
\]
and \eqref{eq:def_mti} takes the form
\be\label{eq:44.03}
\widetilde{M}(z)=\frac{(z^2+1)^{k+1}}{k!^2}\Big(\psi^{(k)}(\I), \ \big((\widetilde{H}-z)^{-1}-\mathcal{R}\big)\psi^{(k)}(\I)\Big).
\ee
Observe that $\widetilde{M}(z)$ is analytic in $\C\setminus\sig(H)$ and, moreover, $\widetilde{M}\in N^\infty_{\gk_l}$.

As in the previous subsections, we set
\be\label{eq:44.09}
Q_{n_l}(z,x):=\frac{(z^2+1)^{k+1}}{k!^2}\int_x^\infty\Big(\big((\widetilde{H}-z)^{-1}-\mathcal{R}\big)\psi^{(k)}(\I,t)\Big)\overline{\psi^{(k)}(\I,t)} dt.
\ee
Note that
\be\label{eq:44.10}
\widetilde{M}(z)=\lim_{x\to 0}Q_{n_l}(z,x).
\ee
Observe that
\begin{align*}
Q_{n_l}(z,x)=Q_{n_l-1}(z,x)-(z^2+1)^{k+1}\re \Big(Q_{n_l-1}(\I,x)\Big)\\
=Q_{n_l-1}(z,x)-(z^2+1)\re\Big(\int_x^\infty \psi^{(k+1)}(\I,t)\overline{\psi^{(k)}(\I,t)}dt\Big).
\end{align*}
Arguing as in the previous subsection and using Corollary \ref{cor:wronski} with the representations from Lemmas \ref{lem_5.5} and \ref{lem:theta_fr}, after straightforward calculation we arrive at the following relation
\be
\widetilde{M}(z)=M(z)-P_M(z)-(F(z)-P_F(z)),
\ee
where $F$ is a real entire function from Lemma~\ref{lem:theta_fr} and $P_f$ is a real polynomial of order at most $n_l=2k+2$ such that
\be\label{eq:44.17}
P_f^{(j)}(\I)=f^{(j)}(\I),\qquad j\in\{0,\dots,k+1\}.
\ee
Thus we proved the following result.
\begin{lemma}\label{lem_m=m_2k+2}
Let $l\in[2k+3/2,2k+5/2)$ and assume Hypothesis \ref{hyp_q}. Let the functions $M$ and $\widetilde{M}$ be defined by \eqref{m-tau1} and \eqref{eq:44.03}, respectively. Then
\be\label{eq:43.19}
M(z)=\widetilde{M}(z)+P_M(z)+G(z),\quad z\in \C_+,
\ee
where the function $G$ is entire. Moreover, $G(z)=F(z)-P_F(z)$, where the function $F$ is given in Lemma~\ref{lem:theta_fr} and $P_f(z)$ is a real polynomial of order at most $n_l=2k+2$ satisfying \eqref{eq:44.17}.
\end{lemma}

\begin{proof}[Proof of Theorem \ref{th:main} in the case $n_l=2k+2$, $k\in\N$]
The first part is contained in Lemma \ref{lem_m=m_2k+2}.

Further, by Lemma~\ref{lem_m=m_2k+2} and \eqref{eq:b.08},  $M(z)$ is an $N_{\gk_l}^\infty$-function if and only if $F(z)=\sum_{j=0}^m a_mz^m$ satisfies the conditions \eqref{eq:b.08} with $\kappa=\kappa_l$.

In particular, $M(z)\in N_{\gk_l}$ if $\theta(z,x)$ is a Frobenius type solution.
\end{proof}

\appendix

\section{Hardy inequality}
\label{app:hi}

Let $l> -1$. Define kernels
\be
K_l(x,y) := \begin{cases} x^{-(l+1)} y^l , & y \leq x,\\ 0, & y>x, \end{cases}
\ee
and associated integral operators
\be\label{hi_01k}
(\mathcal{K}_l f)(x):= \int_0^\infty K_l(x,y)f(y)dy = \frac{1}{x^{l+1}}\int_0^x y^l f(y)dy,
\ee
\be\label{hi_02k}
(\hat{\mathcal{K}}_l g)(y):= \int_0^\infty K_l(x,y) g(x)dx = y^l \int_y^\infty x^{-l-1} g(x)dx.
\ee
First of all we will need the following elementary facts:
\be\label{Kllog1}
\mathcal{K}_l(\log)(x) = \frac{\log(x)}{l+1} - \frac{1}{(l+1)^2}, \quad l> -1,
\ee
and
\be\label{Kllog2}
\hat{\mathcal{K}}_l(\log)(x) = \frac{\log(x)}{l} + \frac{1}{l^2}, \quad l>  0.
\ee
Furthermore, by Theorem 319 of \cite{hlp}, the following inequalities hold
\begin{align}\label{hi_01}
\|\mathcal{K}_l f \|_p &\leq \frac{p}{p(l+1)-1} \|f\|_p, \quad f\in L^p(0,\infty),\\ \label{hi_02}
\|\hat{\mathcal{K}}_l g\|_q &\leq \frac{q}{q l+1} \|g\|_q, \quad g\in L^q(0,\infty),
\end{align}
for $p\in (1,\infty)$, $\frac{1}{p}+\frac{1}{q}=1$ and $(l+1) p > 1$ (resp.\ $l q > -1$).
Moreover, H\"older's inequality implies
\begin{align}\label{hi_sup1}
|(\mathcal{K}_l f)(x)| &\le \frac{x^{1/q-1}}{(1+l q)^{1/q}} \|f\|_p, \quad l > - \frac{1}{q},\\ \label{hi_sup2}
|(\hat{\mathcal{K}}_l g)(y)| & \le \frac{y^{1/p-1}}{(p(l+1)-1)^{1/p}} \|g\|_q, \quad l > \frac{1}{p} -1.
\end{align}

\begin{lemma}\label{lem:Kl}
Let $a>0$ and $l>-1$.
The operator $\mathcal{K}_l$ is a bounded operator in $L^p(0,a)$ satisfying
\be\label{hi_01a}
\|\mathcal{K}_l f \|_p \leq \frac{p}{p(l+1)-1} \|f\|_p, \quad f\in L^p(0,a),
\ee
for any $p\in(\frac{1}{l+1},\infty]$ if $-1 < l \le 0$ and any $p\in[1,\infty]$ if $l>0$.
Moreover, if $f\in C[0,a]$, then $\mathcal{K}_l(f) \in C[0,a]$ with
\be\label{hi_01contin}
\lim_{x\to0} \mathcal{K}_l(f)(x) = \frac{f(0)}{l+1}.
\ee

Similarly, the operator $\hat{\mathcal{K}}_l$ is a bounded operator in $L^p(0,a)$ satisfying
\be\label{hi_02a}
\|\hat{\mathcal{K}}_l f\|_p \leq \frac{p}{p l+1} \|f\|_p, \quad f\in L^p(0,a),
\ee
for any $p\in[1,\frac{1}{-l})$ if $-1<l\le 0$ and any $p\in[1,\infty]$ if $l>0$.
Moreover, if $l>0$ and $f\in C[0,a]$, then $\hat{\mathcal{K}}_l(f) \in C[0,a]$ with
\be\label{hi_02contin}
\lim_{x\to0} \hat{\mathcal{K}}_l(f)(x) = \frac{f(0)}{l}.
\ee
\end{lemma}

\begin{proof}
Equation \eqref{hi_01a} follows from \eqref{hi_01} except for the boundary cases.
The case $p=\infty$ is trivial. For the case $p=1$ if $l>0$ consider bounded functions (which are dense) and take the limit $p\to 1$ in
$\eqref{hi_01a}$. Finally, \eqref{hi_01contin} follows from l'H\^{o}pital's rule.
Equation \eqref{hi_02a} is proven similar.
\end{proof}

Moreover, we will also need the case of Sobolev spaces $W^{1,p}(0,a)$. Recall that the norm of $f\in W^{1,p}$ is defined by
$\|f\|_{W^{1,p}}=\|f\|_{L^p}+\|f'\|_{L^p}$.

\begin{lemma}\label{lem:KlW}
Let $a>0$ and $l>-1$.
The operator $\mathcal{K}_l$ is a bounded operator in $W^{1,p}(0,a)$ viz.\
\be\label{hi_01B}
\|\mathcal{K}_l f \|_{W^{1,p}} \leq C_l \|f\|_{W^{1,p}}, \quad f\in W^{1,p}(0,a),
\ee
for any $p\in[1,\infty]$. Moreover,
\be\label{hi_lim1}
\lim_{x\to 0}x(\mathcal{K}_lf)'(x)= \frac{1}{l+1}\lim_{x\to 0} x f'(x)
\ee
whenever the limit on the right-hand side exists.

Similarly, the operator $\hat{\mathcal{K}}_l$ is bounded in $W^{1,p}(0,a)$ viz.\
\be\label{hi_02B}
\|\hat{\mathcal{K}}_l f \|_{W^{1,p}} \leq \hat{C}_l \|f\|_{W^{1,p}}, \quad f\in W^{1,p}(0,a),
\ee
 for any $p\in[1,\frac{1}{1-l})$ if $0<l\le 1$ and any $p\in[1,\infty]$ if $l> 1$.
Moreover,
\be\label{hi_lim2}
\lim_{x\to 0}x(\hat{\mathcal{K}}_l f)'(x)=\frac{1}{l} \lim_{x\to 0} x f'(x)
\ee
whenever the limit on the right-hand side exists.
\end{lemma}

\begin{proof}
Integrating by parts,
\[
\mathcal{K}_l(f)(x) =\frac{1}{(l+1)}\left(f(x) - x^{-l-1}\int_0^x y^{l+1} f'(y) dy\right),
\]
we get
\be
(\mathcal{K}_lf)'(x) = (\mathcal{K}_{l+1}f')(x)
\ee
and the first claim follows from \eqref{hi_01a}. Equation \eqref{hi_lim1} follows again from l'H\^{o}pital's rule.

The second part is similar using
\be\label{hi_der02}
(\hat{\mathcal{K}}_l f)'(x) = (\hat{\mathcal{K}}_{l-1} f')(x).
\ee
\end{proof}

Concerning $l=0$, we note
\be\label{hi_03a}
(\hat{\mathcal{K}}_0 f)(x)=\int_x^a y^{-1} f(y) dy = -f(0)\log(x/a)+\int_x^a (\mathcal{K}_0 f')(y) dy
\ee
and hence, by \eqref{hi_01a}, the operator
\be\label{hi_03b}
(\widetilde{\mathcal{K}}_0 f)(x) := (\hat{\mathcal{K}}_0 f)(x) + f(0)\log(x)
\ee
is bounded on $W^{1,p}(0,a)$ for $p\in(1,\infty]$. Moreover,
\be\label{hi_03c}
\lim_{x\to 0} x (\hat{\mathcal{K}}_0 f)'(x) = - f(0).
\ee

To cover also the case $p=1$ we note
\be\label{hi_log}
\|\mathcal{K}_{\log}(f)\|_1 \leq \|f\|_1,\quad \mathcal{K}_{\log}(f):=\frac{1}{x}\int_0^x\frac{f(y)}{1-\log(y/a)} dy.
\ee
This follows from the next lemma upon choosing $I(x)=1-\log(x/a)$.

\begin{lemma}\label{lem:a_I}
Let $I(x)\in AC(0,a]$ with $I'(x) \le 0$ and (w.l.o.g.) $I(a)=1$. Consider
\[
(\mathcal{K}_{I} f)(x) := -I'(x) \int_0^x \frac{f(y)}{I(y)} dy, \qquad f\in L^1(0,a).
\]
Then
\[
\|\mathcal{K}_{I}(f)\|_1 \leq \|f\|_1.
\]
\end{lemma}

\begin{proof}
Using integration by parts we obtain
\begin{align*}
\| \mathcal{K}_I f \|_1 & = \int_0^a \left| -I'(x) \int_0^x \frac{f(y)}{I(y)} dy \right| dx
\le -\int_0^a I'(x) \int_0^x \frac{|f(y)|}{I(y)} dy\, dx\\
&= - \int_0^x \frac{I(x)}{I(y)} |f(y)| dy \Big|_0^a + \int_0^aI(y) \frac{|f(y)|}{I(y)} dy\\
&= \int_0^a \left(1 - \frac{1}{I(y)}\right) |f(y)| dy \le \|f\|_1.
\end{align*}
\end{proof}

\section{Generalized Nevanlinna functions}\label{app:nkappa}

In this appendix we collect some information on the classes $N_\kappa$
of generalized Nevanlinna functions \cite{krlan}. By $N_\kappa$,  $\kappa \in \N_0$, we denote the set of
all functions $M(z)$ which are meromorphic in $\C_+\cup \C_-$, satisfy the symmetry condition
\be\label{eq:B.01}
M(z)=M(z^*)^*\quad
\ee
for all $z$ from the domain $\mathcal{D}_M$ of holomorphy of $M(z)$,
and for which the Nevanlinna kernel
\be\label{eq:B.02}
\mathcal{N}_M(z,\zeta)=\frac{M(z)-M(\zeta)^*}{z-\zeta^*},\quad
z,\zeta\in\mathcal{D}_M,\ \ z\neq\zeta^*,
\ee
has $\kappa$ negative squares. That is,
for any choice of finitely many points $\{ z_j \}_{j=1}^n \subset \mathcal{D}_M$ the matrix
\be
\left\{ \mathcal{N}_M(z_j,z_k) \right\}_{1\le j,k\le n}
\ee
has  at most $\kappa$ negative eigenvalues and exactly $\kappa$ negative eigenvalues for some choice of $\{z_j\}_{j=1}^n$.
Note that $N_0$ coincides with the class of Herglotz--Nevanlinna functions.

Let $M\in N_\kappa$, $\kappa\ge 1$. A point $\lam_0\in\R$ is said to be a generalized pole of non-positive
type of $M$ if either
\[
\limsup_{\eps\downarrow 0}\eps|M(\lam_0+\I\eps)|=\infty
\]
or the limit
\[
\lim_{\eps\downarrow 0}(-\I\eps)M(\lam_0+\I\eps)
\]
exists and is finite and negative. The point $\lam_0=\infty$ is said to be a generalized pole of non-positive type of $M$
if either
\[
\limsup_{y\uparrow\infty}\frac{|M(\I y)|}{ y}=\infty
\]
or
\[
\lim_{y\uparrow\infty}\frac{M(\I y)}{\I y}
\]
exists and is finite and negative. All limits can be replaced by non-tangential limits.

We are interested in the special subclass $N_\kappa^\infty \subset N_\kappa$ of generalized Nevanlinna function with no
nonreal poles and the only generalized pole of nonpositive type at $\infty$. It follows from Theorem~3.1 (and its proof) and
Lemma~3.3 of \cite{krlan} that

\begin{theorem}\label{th:gnf}
A function $M \in N_\kappa^\infty$ admits the representation
\be \label{minkappa}
M(z) = (1+z^2)^k \int_\R \left(\frac{1}{\lam-z} - \frac{\lam}{1+\lam^2}\right) \frac{d\rho(\lam)}{(1+\lam^2)^k}
+ \sum_{j=0}^l a_jz^j,
\ee
where $k\leq \kappa$, $l\le 2\kappa+1$,
\be\label{minkappa'}
a_j\in\R,
\quad\text{and}\quad
\int_\R(1+\lam^2)^{-k -1}d\rho(\lam)<\infty.
\ee
The measure $\rho$ is given by the Stieltjes--Liv\v{s}i\'{c} inversion formula
\be
\frac{1}{2} \left( \rho\big((\lam_0,\lam_1)\big) + \rho\big([\lam_0,\lam_1]\big) \right)=
\lim_{\eps\downarrow 0} \frac{1}{\pi} \int_{\lam_0}^{\lam_1} \im\big(M(\lam+\I\eps)\big) d\lam.
\ee
The representation \eqref{minkappa} is called irreducible if $k$ is chosen minimal,
that is, either $k=0$ or $\int_\R(1+\lam^2)^{-k} d\rho(\lam)=\infty$.

Conversely, if \eqref{minkappa'} holds, then $M(z)$ defined via \eqref{minkappa} is in $N_\kappa^\infty$
for some $\kappa$. If $k$ is minimal, $\kappa$ is given by:
\be\label{eq:b.08}
\kappa=\begin{cases}
k, & l\le 2k,\\
\floor{\frac{l}{2}}, & l\ge 2k+1,\ l\ \text{even, or},\ l\ \text{odd and}\ a_l>0,\\
\floor{\frac{l}{2}}+1, & l\ge 2k+1,\ l\ \text{odd and},\ a_l< 0.\end{cases}
\ee
\end{theorem}

For additional equivalent conditions we refer to Definition~2.5 in \cite{DLSh}.

Given a generalized Nevanlinna function in $N_\kappa^\infty$, the corresponding $\kappa$
is given by the multiplicity of the generalized pole at $\infty$ which is determined by the facts
that the following limits exist and take values as indicated:
\[
\lim_{y\uparrow \infty} -\frac{M(\I y)}{(\I y) ^{2\kappa-1}} \in (0,\infty],
\qquad
\lim_{y\uparrow \infty} \frac{M(\I y)}{(\I y) ^{2\kappa+1}} \in [0,\infty).
\]
Again the limits can be replaced by non-tangential ones.
This follows from Theorem~3.2 in \cite{lan}. To this end note that if $M(z) \in N_\kappa$, then $-M(z)^{-1}$,
$-M(1/z)$, and $1/M(1/z)$ also belong to $N_\kappa$. Moreover, generalized zeros of $M(z)$ are generalized poles
of $-M(z)^{-1}$ of the same multiplicity.

\begin{lemma}\label{lem:gnf:grow}
Let $M(z)$ be a generalized Nevanlinna function given by \eqref{minkappa}--\eqref{minkappa'} with $l<2k+1$.
Then, for every $0<\gam<2$, we have
\be
\int_\R \frac{d\rho(\lam)}{1+|\lam|^{2k+\gam}} < \infty \quad\Longleftrightarrow\quad
\int_1^\infty \frac{(-1)^k \im(M(\I y))}{y^{2k+\gam}} dy<\infty.
\ee
Concerning the case $\gam=0$ we have
\be
\int_\R \frac{d\rho(\lam)}{(1+\lam^2)^k} = \lim_{y\to\infty} \frac{(-1)^k \im(M(\I y))}{y^{2k-1}},
\ee
where the two sides are either both finite and equal or both infinite.
\end{lemma}

\begin{proof}
The first part follows directly from \cite[\S 3.5]{kk} (see also \cite[Lem.~9.20]{tschroe}). The second part follows by evaluating
the limit on the right-hand side using the integral representation plus monotone convergence (see e.g. \cite[\S 4]{kk}).
\end{proof}

\section{Super singular perturbations}
\label{app:ssp}

In this section we will collect necessary facts on rank one singular perturbations of self-adjoint operators (further details can be found in \cite{DHdS}, \cite{DKSh_05}, \cite{si_95}, see also references therein).

Let $H$ be an unbounded self-adjoint operator in $\hr$. Recall that to every such operator we can assign a scale
of Hilbert spaces $\hr_n$, $n\in\Z$, in the usual way: For $n\ge 0$ set $\hr_n = \dom(|H|^{n/2})$ together with the norm
$\|\psi\|_{\hr_n}=\| (1+|H|^{1/2})^n \psi\|$ and for $n<0$ let $\hr_n$ be the completion of $\hr$ with respect to the norm
$\|\psi\|_{\hr_n}=\| (1+|H|^{1/2})^n \psi\|$. Then the conjugate linear map $\psi \in \hr \subset \hr_{-n} \mapsto \spr{\psi}{.} \in \hr_{-n}^*$
is isometric and we can identify $\hr_{-n}$ with $\hr_n^*$ in a natural way. We will denote the corresponding dual pairing
between $\hr_n^*=\hr_{-n}$ and $\hr_n$ by $\dpr{.}{.}$. Note that $H$ gives rise to a unique extension $\widetilde{H}: \hr_n \to \hr_{n-2}$.

Choose $\varphi\in \hr_{-1}\setminus\hr_{0}$. Consider the following perturbation of $H$,
\be\label{eq_sing_p}
H_\vartheta:=H+\vartheta \dpr{\varphi}{.}\varphi,\qquad \vartheta \in \R\cup\{\infty\},
\ee
where the sum has to be understood as a form sum via the KLMN theorem (see e.g. \cite[Chapter 6.5]{tschroe}).
The operator
\be\label{eq:c.02}
H_{\min}:=H\lceil \ker (\varphi,.)
\ee
is symmetric in $\hr_0$ with deficiency indices $n_\pm(H_{\min})=1$ and the operators $H_\vartheta$ can be considered as a self-adjoint
extensions of $H_{\min}$.

Then the function
\be\label{eq:c.03}
M(z):=\dpr{\varphi}{(\widetilde{H}-z)^{-1}\varphi},\quad z\in \C_+\cup\C_-,
\ee
is well defined for all $z\in\C_+\cup\C_-$ and
is called \emph{the Weyl function} of the symmetric operator $H_{\min}$. It is also \emph{the $Q$-function of the pair $\{H,H_{\min}\}$} in the sense of Krein and Langer \cite{krlan_71}. Namely,
\be\label{eq:c.04}
\frac{M(z)-M(\zeta)}{z-\zeta}= \gamma(\zeta^*)^*\gamma(z)=\dpr{\gamma(\zeta^*)}{\gamma(z)},
\quad \gamma(z):=(\widetilde{H}-z)^{-1}\varphi \in \hr_1,
\ee
where the function $\gamma(z): \C\setminus\R \to\hr_1$ is called \emph{the $\gamma$-field}.
Moreover, the self-adjoint extensions $H_\vartheta$ of $H_{\min}$ can be parameterized via Krein's resolvent formula
\be\label{eq:c.krein_f}
(H_\vartheta-z)^{-1}=(H-z)^{-1}+ \frac{1}{\vartheta^{-1} - M(z)}\dpr{\gamma(z^*)}{.} \gamma(z),\quad z\in\C_+\cup\C_-.
\ee
Note that $M(z)$ is a Herglotz--Nevanlinna function and admits the following representation
\be\label{eq:c.05}
M(z)=c+\int_{\R}\Big(\frac{1}{\lam-z}-\frac{\lam}{1+\lam^2}\Big)d\rho(\lam),
\ee
where $\rho$ is a positive measure on $\R$ satisfying
\be\label{eq:c.06}
\int_{\R}\frac{d\rho(\lam)}{1+\lam^2}<\infty.
\ee
It is well known that the spectral properties of $H$ are closely connected with the properties of $M$\footnote{Without loss of generality we can assume that $H_{\min}$ is simple, i.e., $\hr=\lspan\{\gamma(z): z\in\C_+\cup\C_-\}$.}. Namely,  there is a unitary transformation $U:\hr \to L^2(\R,d\rho)$ such that
$H$ is unitary equivalent to the multiplication operator
\be\label{eq:c.08}
T\widehat{f}=\lam\widehat{f}(\lam),\quad \dom(T)=\Big\{\widehat{f}\in L^2(\R,d\rho):\ \int_\R \lam^2|\widehat{f}(\lam)|^2d\rho(\lam)<\infty\Big\}.
\ee
In particular, the minimal operator $H_{\min}$ is unitary equivalent to
\be\label{eq:c.09}
T_{\min}:=T\lceil\dom(T_{\min}),\quad \dom(T_{\min})=\Big\{\widehat{f}\in \dom(T):\ \int_\R \widehat{f}(\lam)d\rho(\lam)=0\Big\},
\ee
that is, the corresponding unitary operator $U$ maps the boundary condition $\dpr{f}{\varphi}=0$ into $\int_\R \widehat{f}(\lam)d\rho(\lam)=0$.
In particular, the latter means $\widetilde{U}(\varphi)=1$, $U(\gamma(z))=\frac{1}{\lam-z}$. Since $\varphi\in \hr_{-1}\setminus\hr_0$,
we get that $M$ is an $R_0$-function, that is,
\be\label{eq:c.10}
M(z)=\int_\R \frac{d\rho(\lam)}{z-\lam},
\ee
where
\be\label{eq:c.11}
\int_{\R}d\rho(\lam)=\infty,\qquad
\int_{\R}\frac{d\rho(\lam)}{1+|\lam|}<\infty.
\ee
\begin{example}\label{ex:c.01}
Let $q\in L^1_{\mathrm{loc}}(\R_+)$ and $q\in L^1(0,1)$. Let $H_q$ be the Sturm--Liouville operator corresponding to the Neumann boundary condition at $x=0$,
\[
H_q^Nf=\tau f, \quad \tau:=-\frac{d^2}{dx^2}+q(x),\quad \dom(H_q^N)=\{f\in \dom(H_{\max}):\ f'(0)=0\}.
\]
It is also assumed that $\tau$ is limit point at $+\infty$, i.e., the operator $H_q^N$ is self-adjoint in $L^2(\R_+)$.
Setting $\varphi=\delta$, where $\delta$ is the Dirac delta distribution, we find
\[
H_{\min}=-\frac{d^2}{dx^2}+q(x),\quad \dom(H_{\min})=\{f\in\dom(H_{\max}): \ f(0)=f'(0)=0\}.
\]
Let $c(z,x)$ and $s(z,x)$ be entire solutions of $\tau y=z y$ such that $c(z,0)=s'(z,0)=1$ and $c'(z,0)=s(z,0)=0$. The Weyl solution is given by
\[
\psi(z,x)=s(z,x)+m(z)c(z,x)\in L^2(\R_+).
\]
Here $m(z)$ is the Weyl--Titchmarsh m-function.
Clearly,
\[
m(z)=\psi(z,0)=\dpr[L^2]{\delta}{(H_q^N-z)^{-1}\delta}.
 \]
Moreover, the unitary transformation $U$, which maps $H_q^N$ to $T$ defined by \eqref{eq:c.08}, is the usual Fourier transform
\[
\widehat{f}(\lam)=(Uf)(\lam):= \lim_{b\to +\infty}\int_{0}^bf(x)s(\lam,x)dx,
\]
where the right-hand side is to be understood as a limit in $L^2(\R,d\rho)$.
\end{example}

If $\varphi\in \hr_{-2}\setminus\hr_{-1}$, then the operator \eqref{eq_sing_p} can be given a meaning via the extension theory approach as follows:
The operator $H_{\min}$ defined by \eqref{eq:c.02} is symmetric with $n_\pm(H_{\min})=1$ and the Weyl function for $H_{\min}$
(the $Q$-function for the pair $\{H,H_{\min}\}$) can be defined in a similar way, however, appropriate regularization of \eqref{eq:c.03} is needed.
Namely, set
\be\label{eq:c.16}
M(z):=\dpr{\varphi}{\left((\widetilde{H}-z)^{-1}-\mathcal{R}\right)\varphi},\qquad \mathcal{R}=\re \Big( (\widetilde{H}-\I)^{-1}\Big),\quad z\in\C_+\cup\C_-.
\ee
In this case, $M$ is a Herglotz--Nevanlinna function having the form \eqref{eq:c.05}, where the measure satisfies
\be\label{eq:c.17}
\int_{\R}\frac{d\rho(\lam)}{1+|\lam|}=\infty, \qquad \int_{\R}\frac{d\rho(\lam)}{1+\lam^2}<\infty,
\ee
since $\varphi\in \hr_{-2}\setminus\hr_{-1}$.

Let us remark that in the case $\varphi\in \hr_{-2}\setminus\hr_{-1}$ the perturbed operator $H_\vartheta$ is not uniquely defined anymore. It can only be concluded that $H_\vartheta$ coincides with one of the self-adjoint extensions parameterized by the Krein formula
\[
(H_\vartheta-z)^{-1}=(H-z)^{-1}+ \frac{1}{\widetilde{\vartheta}^{-1} - M(z)}\dpr{\gamma(z^*)}{.} \gamma(z),\quad z\in\C_+\cup\C_-.
\]
and additional assumptions on $H$ and $\varphi$ are needed for establishing the connection between $\vartheta$ and $\widetilde{\vartheta}$.

Singular perturbations by $\varphi \in \hr_{-n-1}\setminus\hr_{-n}$ with $n\ge 2$ cannot be treated in terms of the extension theory of the operator $H_{\min}$ in the original space $\hr$ since the operator $H_{\min}$   is essentially self-adjoint in $\hr$, $\overline{H_{\min}}=H=H^*$. However, starting from the pioneering work \cite{be}, there is an
interpretation for the singular perturbations $H_{\vartheta}$ as exit space extensions of an appropriate restriction
of $H$ (see \cite{sho, DHdS, DKSh_05}). These extensions act in a space which is a finite-dimensional extension of
$\hr$. They are non-self-adjoint with respect to the underlying Hilbert space inner
product, but become self-adjoint when a suitable Pontryagin space scalar product
is introduced.

Namely, consider the $\gamma$-field $\gamma(z)=(\widetilde{H}-z)^{-1}\varphi$. Note that $\gamma(z)\notin \hr$ since
$\varphi \in \hr_{-n-1}\setminus\hr_{-n}$ and hence $(\widetilde{H}-z)^{-1}\varphi \in \hr_{-n+1}\setminus\hr_{-n+2}$. To give a sense to the element $\gamma(z)$ and hence to the resolvent formula \eqref{eq:c.krein_f}, let us extend the space $\hr$ by adding the following elements
\be\label{eq:c.12}
\varphi_j:=(\widetilde{H}-\I)^{-j}\varphi,\quad j\in\{1,\dots, k_n\},\qquad k_n:=\floor{n/2}.
\ee
Then the vector
\be\label{eq:c.13}
\gamma(z)=(\widetilde{H}-z)^{-1}\varphi=\sum_{j=1}^{k_n} (z-\I)^{j-1}\varphi_j + (z-\I)^{k_n}(\widetilde{H}-z)^{-1}\varphi_{k_n}
\ee
can be considered as a vector from an extended inner product space $\widetilde{\hr}$ which
contains both $\hr$ and the vectors \eqref{eq:c.12}. In this space the continuation $\widetilde{H}$ of $H$ generates a linear relation $H'$, for which the operator function \eqref{eq:c.13} can be interpreted to form its $\gamma$-field in the sense that
\[
\gamma(z)-\gamma(\zeta)=(z-\zeta)(H'-z)^{-1}\varphi,\quad z,z\in \C_+\cup\C_-.
\]
The inner product $\spr{.}{.}_{\widetilde{\hr}}$ in $\widetilde{\hr}$ should coincide with the form $(.,.)_{\hr}$ generated
by the inner product in $\hr$ if the vectors $u,v$ are in duality, $u\in \hr_{-j}$ and $v\in \hr_{j}$, $j\in \{0,\dots, k_n\}$. For the other vectors in \eqref{eq:c.13} it is supposed
\[
\spr{\varphi_j}{\varphi_i}_{\widetilde{\hr}}=t_{j+i-1},\quad i,j\in \{1,\dots,k_n\},
\]
where $\{t_j\}_{j=0}^{2k_n-1}\subset \R$. The corresponding inner product has precisely $\kappa=k_n$ negative squares (see \cite[\S 4.3]{DKSh_05}). We omit the detailed construction of the exit space as well as the description of extensions (interested reader can find further details in \cite{DHdS, DKSh_05}). Let us only note that one can choose the constants $t_j$ such that the function
\begin{align}
M(z)&:= (z^2+1)^{k_n}\dpr{\varphi_{k_n}}{(\widetilde{H}-z)^{-1}\varphi_{k_n}}\label{eq:c.14}\\
&=(z^2+1)^{k_n}\dpr{(\widetilde{H}-\I)^{-k_n}\varphi}{(\widetilde{H}-z)^{-1}(\widetilde{H}-\I)^{-k_n}\varphi},\nn
\end{align}
if $n=2k_n+1$ and
\begin{align}
M(z)&:= (z^2+1)^{k_n}\dpr{\varphi_{k_n}}{\left((\widetilde{H}-z)^{-1}-\mathcal{R}\right)\varphi_{k_n}},\label{eq:c.15}\\
&\mathcal{R}=\re\Big( (\widetilde{H}-\I)^{-1}\Big)=\frac{1}{2} \left((\widetilde{H}-\I)^{-1}+(\widetilde{H}+\I)^{-1} \right),\nn
\end{align}
if $n=2k_n+2$, is the $Q$--function for $H'$, i.e.,
\[
\frac{M(z)-M(\zeta)}{z-\zeta}=\gamma(\zeta^*)^*\gamma(z)=\spr{\gamma(\zeta^*)}{\gamma(z)}_{\widetilde{\hr}}.
\]
Observe that $M(\cdot)$ is a generalized Nevanlinna function and $M\in N_{k_n}^\infty$. Indeed, since $\varphi_{k_n}\in\hr_{-2}\setminus\hr$, the function
\[
M_0(z)=\frac{M(z)}{(z^2+1)^{k_n}}
\]
admits the representation either \eqref{eq:c.10}--\eqref{eq:c.11} or \eqref{eq:c.16}--\eqref{eq:c.17}. It remains to apply Theorem \ref{th:gnf}.

The function $M(z)$ can be considered as a regularization of the function defined by \eqref{eq:c.03} and
will be called \emph{the singular Weyl function} for the operator $H$. Note that $M(z)$ characterizes the pair $\{H,\ H'\}$ up to unitary equivalence.

\quad

\noindent
{\bf Acknowledgments.}
We thank Vladimir Derkach, Fritz Gesztesy, Annemarie Luger, Mark Malamud, and Alexander Sakhnovich for several helpful discussions.
We also thank the anonymous referee for valuable suggestions improving the presentation of the material.
A.K. acknowledges the hospitality and financial support of the Erwin Schr\"odinger Institute and the financial
support from the IRCSET PostDoctoral Fellowship Program.

\end{document}